\documentclass[12pt,a4paper,reqno]{amsart}
\usepackage{amssymb,amsmath}
\usepackage{amsfonts,amsbsy,bm}
\usepackage[dvipsnames]{xcolor}
\usepackage{latexsym}
\usepackage{exscale}
\usepackage{amsthm}

\usepackage{mathtools}
\usepackage{MnSymbol}

\DeclarePairedDelimiter\floor{\lfloor}{\rfloor}

\usepackage[colorlinks, bookmarks=true]{hyperref}
\usepackage[utf8]{inputenc}

\usepackage[shortlabels]{enumitem}
\usepackage{color}

\newcommand{\R}{\mathbb R}

\newcommand{\N}{\mathbb N}

\newcommand{\F}{\mathbb F}
\newcommand{\B}{\mathcal B}

\newcommand{\be}{{\mathbf e}}

\newcommand {\X} {{\mathbb X}}
\newcommand {\Y} {{\mathbb Y}}

\newcommand {\e} {{\varepsilon}}

\newcommand{\bfe}{{\boldsymbol\e}}

\renewcommand{\phi}{{\varphi}}

\def\supp{\mathop{\rm supp}}

\numberwithin{equation}{section}
\newtheorem*{theorem*}{Theorem}
\newtheorem{theorem}{Theorem}[section]
\newtheorem{lemma}[theorem]{Lemma}
\newtheorem{defi}[theorem]{Definition}
\newtheorem{corollary}[theorem]{Corollary}
\newtheorem{Remark}[theorem]{Remark}
\newtheorem{remark}[theorem]{Remark}
\newtheorem{proposition}[theorem]{Proposition}
\newtheorem{definition}[theorem]{Definition}

\newtheorem{example}[theorem]{Example}
\theoremstyle{definition}

\newcommand{\Ba}[1]{\begin{array}{#1}}
\newcommand{\Ea}{\end{array}}
\newcommand{\Be}{\begin{equation}}
\newcommand{\Ee}{\end{equation}}
\newcommand{\Bea}{\begin{eqnarray}}
\newcommand{\Eea}{\end{eqnarray}}
\newcommand{\Beas}{\begin{eqnarray*}}
\newcommand{\Eeas}{\end{eqnarray*}}
\newcommand{\Benu}{\begin{enumerate}}
\newcommand{\Eenu}{\end{enumerate}}
\newcommand{\Bi}{\begin{itemize}}
\newcommand{\Ei}{\end{itemize}}

\newcommand{\BR}{\begin{Remark} \em}
\newcommand{\ER}{\end{Remark}}
\newcommand{\BE}{\begin{example} \em}
\newcommand{\EE}{\end{example}}

\newcounter{reg}
\newcounter{regTO}
\newcommand{\bff}{\mathbf 1}
\renewcommand\Re{\operatorname{Re}}
\renewcommand\Im{\operatorname{Im}}
\newcommand{\n}{\mathbf{n}}
\newcommand\G{\mathbf{G}}
\newcommand{\K}{\mathbf{K}}
\newcommand{\one}{\mathbf{1}}
\newcommand{\C}{\mathbf{C}}
\usepackage[shortlabels]{enumitem}

\title[Democracy-like properties for sequences with gaps]{Extensions of democracy-like properties for sequences with gaps}

\author[M. Berasategui]{Miguel Berasategui}

\address{Miguel Berasategui
	\\
	IMAS - UBA - CONICET - Pab I, Facultad de Ciencias Exactas y Naturales \\ Universidad de Buenos Aires \\ (1428), Buenos Aires, Argentina}
\email{mberasategui@dm.uba.ar}

\author[P.\ M. Bern\'a]{Pablo M. Bern\'a}
\address{Pablo M. Bern\'a\\
	Departamento de Matem\'atica Aplicada y Estad\'istica, Facultad de Ciencias Econ\'omicas y Empresariales, Universidad San Pablo-CEU, CEU Universities\\ Madrid, 28003 Spain.}
\email{pablo.bernalarrosa@ceu.es}

\begin{document}
\subjclass[2010]{41A65, 41A46, 41A17, 46B15, 46B45.}

\keywords{Non-linear approximation, greedy bases, weak greedy algorithm, quasi-greedy basis.}
\thanks{The first author was supported by ANPCyT  PICT-2018-04104. The second author was supported by Grants PID2019-105599GB-100 (Agencia Estatal de Investigación, Spain) and 20906/PI/18 from Fundaci\'on S\'eneca (Regi\'on de Murcia, Spain).
}

\begin{abstract}
In \cite{O2015}, T. Oikhberg introduced and studied variants of the greedy and weak greedy algorithms for sequences with gaps. In this paper, we extend some of the notions that appear naturally in connection with these algorithms to the context of sequences with gaps. In particular, we will consider sequences of natural numbers for which the inequality $n_{k+1}\le \C n_{k}$ or $n_{k+1}\le \C+n_k$ holds for a positive constant $\C$ and all $k$, and find conditions under which the extended notions are equivalent their regular counterparts. 
\end{abstract}

\maketitle

\section{Introduction}\label{sectionintroduction}

Let $\X$ be a separable, infinite dimensional Banach space over the field $\mathbb F=\mathbb R$ or $\mathbb C$, with dual space $\X^*$. A \textit{fundamental minimal system}  $\B=(\be_i)_{i\in \N}\subset \X$ is a sequence that satisfies the following: 
	\begin{enumerate}[\upshape (i)]
		\item $\X=\overline{[\be_i \colon i\in\N]}$;
		\item there is a (unique) sequence $\B^*=(\be_i^*)_{i=1}^{\infty}\subset \X^*$ of biorthogonal functionals, that is, $\be_k^*(\be_i)=\delta_{k,i}$ for all $k,i\in\N$.
\end{enumerate}
If $\B$ verifies the above conditions and 
$$
\be_i^*(x)=0\qquad \forall i\in\N\Longrightarrow x=0 \qquad\text{(totality)},
$$
we say that $\B$ is a \emph{Markushevich basis}. If there is also a positive constant $\C$ such that
$$\Vert S_m(x)\Vert \leq \C \Vert x\Vert \qquad\forall x\in \X,\,\forall m\in \N,$$ where $S_m$ is the $m$th partial sum $\sum_{i=1}^{m}\be_i^*(x)\be_i$, we say that $\B$ is a \textit{Schauder basis}. Its basis constant $\K$ is the minimum $\C$ for which this inequality holds. \\
If, additionally, there is $\C>0$ such that 
$$\Vert P_A(x)\Vert \leq \C \Vert x\Vert \qquad\forall x\in \X,\, \forall A\subset\N: |A|<\infty,$$ 
where $P_A$ is the projection on $A$ with respect to $\B$ (that is $P_A(x)= \sum_{i\in A}\be_i^*(x)\be_i$), we say that $\B$ is \emph{suppression unconditional}. The suppression unconditionality constant $\C_{su}$ is the minimum $\C$ for which the above holds. Equivalently (though not necessarily with the same constant), $\B$ is \emph{unconditional} if 
$$
\|\sum_{j\in \N}a_j \be_j^*(x)\be_j\|\le \C \|x\|\qquad \forall x\in \X,\,\forall (a_j)_{j\in\N}\subset \F: |a_j|\le 1 \quad\forall j\in \N,
$$
for some $\C>0$.\\
Hereinafter, by a \emph{basis} for $\X$ we mean a fundamental minimal system $\B$ such that both $\B$ and $\B^*$ are semi-normalized, that is 
$$
0<\inf_{i\in \N}\min\{\|\be_i\|,\|\be_i^*\|\}\le \sup_{i\in \N}\max\{\|\be_i\|,\|\be_i^*\|\}<\infty. 
$$
We will use $\B$ to denote a basis, and we define positive constants $\alpha_1, \alpha_2$ as follows: 
$$
\alpha_1:=\sup_{i\in\N}\|\be_i\|\qquad\text{and}\qquad\alpha_2:=\sup_{i\in\N}\|\be_i^*\|.
$$

In 1999, S. V. Konyagin and V. N. Temlyakov introduced the Thresholding Greedy Algorithm (TGA), which has become one of the most important algorithms in the field of non-linear approximation, and has been studied by researchers such as F. Albiac, J. L. Ansorena, S. J. Dilworth, N. J. Kalton, D. Kutzarova, V. N. Temlyakov and P. Wojtaszczyk, among others.  The algorithm essentially chooses for each $x\in \X$ the largest coefficients in modulus with respect to a basis. A relaxed version of this algorithm was introduced by V. N. Temlyakov in \cite{Tem2}.  Fix $t\in (0,1]$.
We say that a set $A(x,t):=A$ is a $t$-\textbf{greedy set} for $x\in\X$ if
$$\min_{i\in A}\vert\be_i^*(x)\vert\geq t \max_{i\not\in A}\vert\be_i^*(x)\vert.$$

\noindent A $t$-\textbf{greedy sum} of order $m$ (or an $m$-term $t$-greedy sum) is the projection
$$\mathbf{G}_m^t(x)=\sum_{i\in A}\be_i^*(x)\be_i,$$
where $A$ is a $t$-greedy set of cardinality $m$. The collection $(\mathbf{G}^t_m)_{m=1}^\infty$ is called the \textbf{Weak Thresholding Greedy Algorithm} (WTGA) (see \cite{Tem,Tem2}), and we denote by $\mathcal G_m^t$ the collection of $t$-greedy sums $\G^t_m$ with $m\in\mathbb N$. If $t=1$, we talk about greedy sets and greedy sums $\G_m$.\\
Different types of convergence of these algorithms have been studied in several papers, for instance \cite{DKK2003, DKKT, KT}. For $t=1$, a central concept in these studies is the notion of quasi-greediness (\cite{KT}). 

\begin{defi}
We say that $\mathcal B$ is quasi-greedy if there exists a positive constant $\mathbf C$ such that
$$\Vert \mathbf{G}_m(x)\Vert \leq \mathbf C\Vert x\Vert,\; \forall x\in\mathbb X, \forall m\in\mathbb N.$$
\end{defi}

The relation between quasi-greediness and the convergence of the algorithm was given by P. Wojtaszczyk in \cite{Wo}: a basis is quasi-greedy if and only 
$$\lim_n \G_n(x)=x,\; \forall x\in\X.$$\\
Recently, T. Oikhberg, in \cite{O2015}, introduced and studied a variant of the WTGA where only the $t$-greedy sums with order in a given increasing sequence of positive integers $\mathbf{n}=(n_k)_{k=1}^\infty$ are considered. In this context, we will consider two types of \emph{gaps} of such a sequence: the \emph{quotient gaps} of the sequence are the quotients $\left(\frac{n_{k+1}}{n_k}\right)_{k}$ when $n_{k+1}>n_{k}+1$, whereas the \emph{additive gaps} of the sequence are the differences $n_{k+1}-n_k$ in such cases (although it is the only sequence without gaps, for the sake of convenience we will allow $\n=\N$ in our proofs and definitions unless otherwise stated). \\
In our context, Oikhberg's central definition is as follows: given $\mathbf{n}=(n_k)_{k=1}^\infty\subset \N$ a strictly increasing sequence $n_1<n_2<...$, a basis $\mathcal B$ is $\n$-$t$-quasi-greedy if
\begin{eqnarray}\label{qO}
\lim_k \G_{n_k}^t(x)=x,
\end{eqnarray}
for any $x\in\X$ and any choice of $t$-greedy sums $\G_{n_k}^t(x)$. In \cite[Theorem 2.1]{O2015}, the author shows that for sequences with gaps there is also a close connection between the boundedness of  t-greedy sums and the convergence of the algorithm. Indeed, $\B$ is $\n$-$t$-quasi greedy if and only of there is $\C>0$ such that 

\begin{eqnarray}\label{q}
\Vert \mathbf{G}^t_n(x)\Vert\leq \C\Vert x\Vert,\; \forall x\in\mathbb X, \forall \G_n^t\in\mathcal G_n^t, \forall n\in\mathbf{n}.
\end{eqnarray}
We will use the notation $\C_{q,t}$ for the minimum $\C$ for which \eqref{q} holds, and we will say that $\B$ is $\C_{q,t}$-$\n$-$t$-quasi-greedy. Of course, if the basis is quasi-greedy, it is $\n$-quasi-greedy for any sequence $\n$ and, moreover, it is also $\n$-$t$-quasi greedy for all $0<t\le 1$ (see \cite[Theorem 2.1]{O2015}, \cite[Lemmas 2.1, 2.3]{KT2002}, \cite[Proposition 4.5]{DKSWo2012}, \cite[Lemma 2.1, Lemma 6.3]{DKO2015}). The reciprocal is false as \cite[Proposition 3.1]{O2015} shows and, in fact, this result shows that for any sequence $\n$ that has \emph{arbitrarily large quotient gaps} (see Definition~\ref{definitionquotientgaps} below), there are Schauder bases that are $\n$-$t$-quasi greedy for all $0<t\le 1$, but not quasi-greedy. On the other hand, it was recently proven that if $\n$ has bounded quotient gaps, a Schauder basis that is $\n$-quasi-greedy is also quasi-greedy (\cite[Theorem 5.2]{BB}). 
\begin{definition}\label{definitionquotientgaps}
	Let $\n=(n_k)_{k\in \N}$ be a strictly increasing sequence of natural numbers with gaps. We say that $\n$ has arbitrarily large quotient gaps if
	$$\limsup_{k\rightarrow\infty}\frac{n_{k+1}}{n_k}=\infty.$$
	
	\noindent Alternatively, for $l\in\mathbb N_{>1}$, we say that $\n$ has $l$-bounded quotient gaps if
$$\frac{n_{k+1}}{n_k}\leq l,$$
for all $k\in\mathbb N$, and we say that it has bounded quotient gaps if it has $l$-bounded quotient gaps for some natural number $l\ge 2$.
\end{definition}
We will also need the following classification: 

\begin{definition}\label{definitiondifferencegaps}
	Let $\n=(n_k)_{k\in \N}$ be a strictly increasing sequence of natural numbers with gaps. We say that $\n$ has arbitrarily large additive gaps if
	$$\limsup_{k\rightarrow\infty}n_{k+1}-{n_k}=\infty.$$
	
	\noindent Alternatively, for $l\in\mathbb N_{>1}$, we say that $\n$ has $l$-bounded additive gaps if $ n_{k+1}-n_{k}\le l$ for all $k\in\mathbb N$, and we say that it has bounded additive gaps if it has $l$-bounded additive gaps for some natural number $l\ge 2$.
\end{definition}

Several properties that appear naturally in connection to these algorithms have been studied in the literature. In \cite{KT}, Konyagin and Temlyakov characterized greedy bases (that is, bases where the greedy algorithm produces the best possible approximation) as those that are unconditional and \textit{democratic}, where democratic bases are those bases such that there is $\C>0$ such that 
$$
\|\sum_{j\in A}\be_j\|\le \C\|\sum_{j\in B}\be_j\|\qquad\forall A,B\subset\N: |A|\le |B|<\infty. 
$$ 
A similar characterization was proven in \cite{DKKT} for almost greedy bases, which are quasi-greedy and democratic. In papers such as \cite{AA2017, AW, DKOSZ, Wo}, the authors studied properties such as symmetry for largest coefficients - which has been used to characterize $1$-almost greediness $1$-greediness - and unconditionality for constant coefficients - which is used for example to characterize superdemocracy.   \\
Here, motivated by the theory introduced by Oikhberg in \cite{O2015} and by some of the examples from \cite{BBG} and \cite{BBGHO}, we extend some of the aforementioned concepts to the context of sequences with gaps, and study their relations with their standard counterparts, that is the notions for $\n=\N$. 
 \\This paper is organized as follows: in Section~\ref{sectionnuccc} we introduce and study the notions of $\n$-unconditionality for constant coefficients and the $\n$-UL property. Section~\ref{sectiondemocracy} focuses on the concepts of $\n$-democracy and other democracy-like properties, whereas Section \ref{sectionnsymmetry} looks at $\n$-symmetry for largest coefficients and closely related properties. In Section~\ref{sectionexamples}, we consider two families of examples that are used throughtout the paper. \\
We will use the following notation throughout the paper - in addition to that already introduced: for $A$ and $B$ subsets of $\mathbb N$, we write $A<B$ to mean that  $\max A<\min B$. If $m\in \mathbb N$, we write $m < A$ and $A < m$ for $\{ m\} < A$ and  
$A <\{ m\}$ respectively (and we use the symbols ``$>$'', ``$\ge$'' and ``$\le$'' similarly). Also, $A\cupdot B$ means the union of $A$ and $B$ with $A\cap B=\emptyset$, and $\N_{>k}$ means the set $\N\setminus \lbrace 1,\dots,k\rbrace$. 

For $A\subset\N$ finite and a basis $\B$, $\Psi_A$ denotes the set of all collections of sequences $\bfe = (\e_n)_{n\in A}\subset \mathbb F$ such that $\vert \varepsilon_n\vert=1$ and
$$
\bff_{\bfe A}[\mathcal B,\X]:=\bff_{\bfe A}=\sum_{n\in A}\e_n \be_n.
$$
If $\bfe\equiv 1$, we just write $\bff_A$. Also, every time we have index sets $A\subset B$ and $\bfe\in \Psi_B$, we write $\bff_{\bfe A}$ considering the  natural restriction of $\bfe$ to $A$, with the convention that $\bff_{\bfe A}=0$ if $A=\emptyset$.\\
As usual, by $\supp{(x)}$ we denote the support of $x\in \X$, that is the set $\{i\in \mathbb{N}: \be_i^*(x)\not=0\}$. For $x\in X$ and $1\le p\le \infty$, by $\|x\|_{p}$ we mean the $\ell_p$-norm of $(\be_i^*(x))_{i}$ when it is well-defined. Finally, we set 
\begin{equation*}
\kappa:=
\begin{cases}
1 & \text{ if } \F=\R;\\
2 & \text{ if } \F=\mathbb{C}.
\end{cases}
\end{equation*}

\section{Unconditionality for constant coefficients} \label{sectionnuccc}
In the literature, it is well known that every quasi-greedy basis is unconditional for constant coefficients, that is, for every finite set $A$ and every sequence of signs $\e\in\Psi_A$,
$$\Vert\one_{\e A}\Vert \approx \Vert\one_A\Vert.$$
This condition was introduced by P. Wojtaszczyk in \cite{Wo} and it is the key to characterize superdemocracy using democracy (see for instance \cite[Lemma 3.5]{BBG} for more details), among other applications. Here, we consider a natural extension for sequences with gaps.

\begin{definition}\label{nuccc} We say that $\B$ is $\n$-unconditional for constant coefficients if there is $\C>0$ such that 
\begin{eqnarray}\label{ucc}
	\|\bff_{\bfe A}\|\le \C\|\bff_{\bfe' A}\|
\end{eqnarray}
	for all $A\subset \N$ with $|A|\in\n$ and all $\bfe, \bfe'\in \Psi_A$. The smallest constant verifying \eqref{ucc} is denoted by $\mathbf{K}_{u}$ and we say that $\mathcal B$ is $\K_u$-$\n$-unconditional for constant coefficients. If $\n=\mathbb N$, we say that $\mathcal B$ is $\mathbf{K}_u$-unconditional for constant coefficients.
\end{definition}

The following result gives sufficient conditions under which $\n$-unconditionality for constant coefficients entails the classical notion - and then, it is equivalent to it. 

\begin{proposition}\label{propositionucccboundedgaps}Let $\B$ be a basis that is $\K_u$-$\mathbf{n}$-unconditional for constant coefficients. Then:
\begin{enumerate}[i)]
\item\label{boundedlineargapsucc}If $\n$ has $l$-bounded additive gaps, $\B$ is $\C$-unconditional for constant coefficients with $\C\le \max\{(n_1-1)\alpha_1\alpha_2, \K_u+\K_u l\alpha_1\alpha_2 +l\alpha_1\alpha_2  \}$.
\item \label{boundedquotientgapssucc}If $\n$ has $l$-bounded quotient gaps and $\B$ is Schauder with constant $\K$, then $\B$ is $\C$-unconditional for constant coefficients with $\C\leq \max\lbrace (n_1-1)\alpha_1\alpha_2, (2l-1)\K_u\K\rbrace. $
\end{enumerate}
\end{proposition}
\begin{proof} 
\ref{boundedlineargapsucc} 
Fix a finite set $A\subset\N$ with $|A|\not\in \mathbf{n}$, and $\bfe, \bfe' \in \Psi_A$. If $|A|<n_1$, we have
	\begin{equation}
		\|\bff_{\bfe A}\|\le |A|\alpha_1\le (n_1-1)\alpha_1\le (n_1-1)\alpha_1\alpha_2\|\bff_{\bfe' A}\|.\label{smallercardinal39}
	\end{equation}
If $|A|>n_1$, let 
$$
k_0:=\max_{k\in \N}\{n_k<|A|\},
$$
and choose $A_1\subset A$ with $|A_1|=n_{k_0}$. We have 
\begin{eqnarray*}
\|\bff_{\bfe A}\|&\le&\|\bff_{\bfe A_1}\|+\|\bff_{\bfe A\setminus A_1}\|\le \K_u\|\bff_{\bfe' A_1}\|+l \alpha_1\alpha_2  \|\bff_{\bfe' A}\|\\
&\le&  \K_u\|\bff_{\bfe' A}\|+\K_u \|\bff_{\bfe' A\setminus A_1}\|  +l \alpha_1\alpha_2  \|\bff_{\bfe' A}\|\le (\K_u+\K_u l\alpha_1\alpha_2 +l\alpha_1\alpha_2)\|\bff_{\bfe' A}\|,
\end{eqnarray*}
which, when combined with \eqref{smallercardinal39}, gives \ref{boundedlineargapsucc}. \\
\ref{boundedquotientgapssucc} Fix $A\subset\N, \bfe, \bfe' \in \Psi_A$ as before. If $|A|<n_1$, then by the same argument given above we have \eqref{smallercardinal39}. On the other hand, if $|A|>n_1$, define $k_0$ as above. Since $\mathbf{n}$ has $l$-bounded quotient gaps and $n_{k_0}<|A|<n_{k_0+1}\le l n_{k_0}$, there is $2\le m\le l$ and a partition of $A$ into nonempty disjoint sets $(A_j)_{1\le j\le m}$ such that
	\begin{equation}
		|A_1|\le n_{k_0},\qquad |A_j|=n_{k_0}\forall 2\le  j\le m,\qquad A_j<A_{j+1}\forall 1\le j\le m-1.\nonumber
	\end{equation}
	For each $2\le j \le m$, we get 
	\begin{eqnarray}
		\|\bff_{\bfe A_j}\|&\le& \K_u\|\bff_{\bfe' A_j}\|\le 2\K_u\K\|\bff_{\bfe' A}\|. \label{rightbounb39}
	\end{eqnarray}
	Let $B$ be the (perhaps empty) set consisting of the first $n_{k_0}-|A_1|$ elements of $A\setminus A_1$. We have
	$$
	\|\bff_{\bfe A_1}\|\le \max_{\epsilon\in \{-1,1\}}\|\bff_{\bfe A_1}+\epsilon \bff_{B}\|\le \K_u\|\bff_{\bfe' A_1}+ \bff_{\bfe' B}\|\le \K_u\K\|\bff_{\bfe' A}\|. 
	$$
	From this and \eqref{rightbounb39}, it follows by the  triangle inequality that 
\begin{equation}
\|\bff_{\bfe A}\|\le (2l-1)\K_u \K\|\bff_{\bfe' A}\|.\label{biggercardinal39}
\end{equation}
The proof is completed combining \eqref{smallercardinal39} and \eqref{biggercardinal39}.
\end{proof}

In the case $\n=\N$, it is is known that quasi-greediness implies a property that is stronger than unconditionality for constant coefficients, namely the UL property: if $A$ is a finite set, then for any sequence $(a_i)_{i\in A}$,
\begin{equation}\label{ul}
\min_{i\in A}|a_i|\|\mathbf{1}_{A}\|\lesssim \|\sum\limits_{i\in A}a_i\be_i\|\lesssim \max_{i\in A}|a_i|\|\mathbf{1}_{A}\|.
\end{equation}
This relation was shown for the first time in \cite{DKKT} when $\F=\R$ and, for the complex case, the result was proved in \cite{Bt}. Moreover, the UL property has own life since in \cite[Section 5.5]{BBG}, the authors gave the first example in the literature of a basis in a Banach space such that \eqref{ul} is satisfied but the basis is not quasi-greedy. Now, we extend this notion to the context of sequences with gaps.

\begin{definition}\label{definitionnUL} We say that a basis $\mathcal B$ has the $\mathbf{n}$-UL property if there are positive constants $\C_1,\C_2$ such that
	\begin{equation}
	\frac{1}{\C_1}\min_{i\in A}|a_i|\|\mathbf{1}_{A}\|\le \|\sum\limits_{i\in A}a_i\be_i\|\le \C_2\max_{i\in A}|a_i|\|\mathbf{1}_{A}\|\label{nul}
	\end{equation}
	for all $A\subset \N$ with $|A|\in \mathbf{n}$ and all scalars $(a_i)_{i\in A}$. If $\n=\mathbb N$, we say that $\mathcal B$ has the UL property with constants $\C_1$ and $\C_2$.
\end{definition}

For sequences with (in either sense) bounded gaps, we have the following result, similar to Lemma~\ref{propositionucccboundedgaps}.

\begin{proposition}\label{propositionDboundedgaps}Let $\B$ be a basis that has the $\n$-UL property with constants $\C_1$ and $\C_2$. The following hold: 
\begin{enumerate}[i)]
\item \label{nUL>ULbounded} If $\B$ has $l$-bounded additive gaps, $\B$ has the UL property with constants $\C_1', \C_2'$ verifying the following bounds:
$$
\C_1'\le \max\{(n_1-1)\alpha_1\alpha_2,\C_1+l\alpha_1\alpha_2+\C_1 l\alpha_1\alpha_2\},
$$
and 
$$
\C_2'\le \max\{(n_1-1)\alpha_1\alpha_2,\C_2+l\alpha_1\alpha_2+\C_2 l\alpha_1\alpha_2\}.
$$
\item \label{nUL>ULboundedSchauder} If $\B$ has $l$-bounded quotient gaps and $\B$ is Schauder with constant $\K$, $\B$ has the UL property with constants $\C_1', \C_2'$ verifying the following bounds:
$$
\C_1'\le \max\{(n_1-1)\alpha_1\alpha_2,\K^2\C_1+2(l-1)\C_1\K\},
$$
and 
$$
\C_2'\le \max\{(n_1-1)\alpha_1\alpha_2,\K^2\C_2+2(l-1)\C_2\K\}.
$$
\end{enumerate}
\end{proposition}
\begin{proof}
\ref{nUL>ULbounded}	Fix a finite set $A\subset \N$ with $|A|\not\in \mathbf{n}$, and scalars $(a_i)_{i\in A}$. If $|A|<n_1$, then 
\begin{equation}
\min_{i\in A}|a_i|\|\bff_{A}\|\le \min_{i\in A}|a_i||A|\alpha_1\le (n_1-1)\alpha_1\alpha_2\|\sum\limits_{i\in A}a_i\be_i\|,\label{onesidesmallA}
\end{equation}
and 
\begin{equation}
\|\sum\limits_{i\in A}a_i\be_i\|\le \max_{i\in A}|a_i|n_1\alpha_1\le (n_1-1)\alpha_1\alpha_2\|\bff_{A}\|.\label{anothersidesmallA}
\end{equation}

On the other hand, if $|A|>n_1$, let 
	$$
	k_0:=\max_{k\in \N}\{n_k<|A|\},
	$$
and choose $A_1\subset A$ a greedy set with $|A_1|=n_{k_0}$. We have 
\begin{eqnarray}
\min_{i\in A}|a_i|\|\bff_{A}\|&\le& \min_{i\in A}|a_i|\|\bff_{A_1}\|+\min_{i\in A}|a_i|\|\bff_{A\setminus A_1}\|\le \C_1\left\Vert\sum_{i\in A_1}a_i\be_i\right\Vert+l\alpha_1\alpha_2 \left\Vert\sum_{i\in A}a_i\be_i\right\Vert\nonumber\\
&\le& (\C_1+l\alpha_1\alpha_2)\left\Vert\sum_{i\in A}a_i\be_i\right\Vert+\C_1\left\Vert\sum_{i\in A\setminus A_1}a_i\be_i\right\Vert\nonumber\\
&\le& (\C_1+l\alpha_1\alpha_2)\left\Vert\sum_{i\in A}a_i\be_i\right\Vert+\C_1 l\alpha_1\max_{i\in A}|a_j|\nonumber\\
&\le& (\C_1+l\alpha_1\alpha_2+\C_1 l\alpha_1\alpha_2)\left\Vert\sum_{i\in A}a_i\be_i\right\Vert, \label{onesidebigA}
\end{eqnarray}
and 
\begin{eqnarray}
\left\Vert\sum_{i\in A}a_i\be_i\right\Vert&\le& \left\Vert\sum_{i\in A_1}a_i\be_i\right\Vert+\left\Vert\sum_{i\in A\setminus A_1}a_i\be_i\right\Vert\le \C_2\max_{i\in A_1}|a_i|\|\bff_{A_1}\|+l\alpha_1\max_{i\in A\setminus A_1}|a_i|\nonumber\\
&\le&  \C_2\max_{i\in A_1}|a_i|\|\bff_{A}\|+\C_2\max_{i\in A_1}|a_i|\|\bff_{A\setminus A_1}\|+l\alpha_1\alpha_2\max_{i\in A\setminus A_1}|a_i|\|\bff_{A}\|\nonumber\\
&\le& \max_{i\in A}|a_i|(\C_2+\C_2l\alpha_1\alpha_2+l\alpha_1\alpha_2)\|\bff_{A}\|.\label{anothersidebigA}
\end{eqnarray}
The proof of \ref{nUL>ULbounded} is completed combining \eqref{onesidesmallA}, \eqref{anothersidesmallA}, \eqref{onesidebigA}, and \eqref{anothersidebigA}.\\	
\ref{nUL>ULboundedSchauder}  Fix a finite set $A\subset \N$ with $|A|\not\in \mathbf{n}$, and scalars $(a_i)_{i\in A}$. The case $|A|<n_1$ is handled as in the proof of \ref{nUL>ULbounded}, so we assume $|A|>n_1$ and we set $k_0$ as before. Since $\mathbf{n}$ has $l$-bounded quotient gaps, there is $2\le m\le l$ and a partition of $A$ into nonempty  sets $(A_j)_{1\le j\le m}$ such that
	\begin{equation}
	|A_1|\le n_{k_0},\qquad |A_k|=n_{k_0}\forall 2\le  j\le m,\qquad\text{and}\qquad A_j<A_{j+1}\forall 1\le j\le m-1.\nonumber
	\end{equation}
	For each $2\le j \le m$, applying the $\mathbf{n}$-UL property and the Schauder condition we get 
	
	\begin{eqnarray}
	\min_{i\in A}|a_i|\|\bff_{A_j}\|&\le& \min_{i\in A_j}|a_i|\|\bff_{A_j}\|\le  \C_1\|\sum\limits_{i\in A_j}a_i\be_i\|\le 2\C_1\K\|\sum\limits_{i\in A}a_i \be_i\|. \label{rightbounb3}
	\end{eqnarray}
	Let $B$ be the set consisting of the first $n_{k_0}$ elements of $A$. Since $A_1$ is the set consisting of the first $|A_1|\le n_{k_0}$ elements of $A$, we have
	$$
	\min_{i\in A}|a_i|\|\bff_{A_1}\|\le \min_{i\in B}|a_i|\K\|\bff_{B}\|\le \K \C_1\|\sum\limits_{i\in B}a_i\be_i\|\le \K^2\C_1\|\sum\limits_{i\in A}a_i\be_i\|. 
	$$
Combining this with \eqref{rightbounb3}, it follows by the triangle inequality that 
	\begin{equation}
	\min_{i\in A}|a_i|\|\bff_{A}\|\le (\K^2\C_1+2(l-1)\C_1\K)\|\sum\limits_{i\in A}a_i\be_i\|. \label{biggercardinal3}
	\end{equation}
Similarly,  for each $2\le j \le m$, applying the $\mathbf{n}$-UL and Schauder conditions we get 
	\begin{eqnarray}
	\|\sum\limits_{i\in A_j}a_i\be_i\|&\le& \C_2\max_{i\in A_j}|a_i|\|\bff_{A_j}\|
\le 2\K \C_2\max_{i\in A}|a_i|\|\bff_{A}\|, \label{rightbounb9}
	\end{eqnarray}
and
	$$
	\|\sum\limits_{i\in A_1}a_i\be_i\|\le \K\|\sum\limits_{i\in B}a_i\be_i\|\leq \K \C_2\max_{i\in B}|a_i|\|\bff_B\|\le \K^2\C_2\max_{i\in A}|a_i|\|\bff_A\|. 
	$$
	From this and \eqref{rightbounb9}, by the triangle inequality we obtain 
	$$
	\|\sum\limits_{i\in A}a_i\be_i\|\le (\K^2\C_2+2(l-1)\K \C_2)\max_{i\in A}|a_i|\|\bff_A\|. 
	$$
The proof is completed combining the above inequality with \eqref{onesidesmallA}, \eqref{anothersidesmallA} and \ref{biggercardinal3}. 
\end{proof}

Propositions~\ref{propositionucccboundedgaps} and~\ref{propositionDboundedgaps} give sufficient conditions on a sequence with gaps $\n$ under which $\n$-unconditionality for constant coefficients and the $\n$-UL property are equivalent to their standard counterparts. Our next result shows that these conditions are also necessary.

\begin{proposition}\label{propositionboundedgapsnULnoUl} Let $\n$ be a sequence. The following hold: 
\begin{itemize}
\item If $\n$ has arbitrarily large additive gaps, there is a Banach space $\X$ with a Markushevich basis $\B$ that has the $\n$-UL property, but is not unconditional for constant coefficients.
\item If $\n$ has arbitrarily large quotient gaps, there is a Banach space $\X$ with a Schauder basis $\B$ that has the $\n$-UL property, but is not unconditional for constant coefficients.
\end{itemize}
\end{proposition}
\begin{proof}
See Examples~\ref{examplendemocracylargelineargaps} and~\ref{examplendemocratic}.
\end{proof}

Summing up, we have the following equivalences. 
\begin{corollary}\label{corollaryequivschauderul}Let $\n$ be a sequence with gaps. The following are equivalent:
\begin{enumerate}[i)]
\item \label{boundedquot}$\n$ has bounded quotient gaps. 
\item \label{schaudernucc} Every Schauder basis that is $\n$-unconditional for constant coefficients is  unconditional for constant coefficients. 
\item \label{schaudernul} Every Schauder basis that has the $\n$-UL property has the UL property.
\end{enumerate}
\end{corollary}

\begin{corollary}\label{corollaryequivmarkul}Let $\n$ be a sequence with gaps. The following are equivalent:
\begin{enumerate}[i)]
\item \label{boundedlinearequiv}$\n$ has bounded additive gaps. 
\item \label{nucc>ucc} Every basis that is $\n$-unconditional for constant coefficients is  unconditional for constant coefficients. 
\item \label{Mnucc>ucc} Every Markushevich basis that is $\n$-unconditional for constant coefficients is  unconditional for constant coefficients. 
\item \label{nUL>UL} Every basis that has the $\n$-UL property has the UL property.
\item \label{MnUL>UL} Every Markushevich basis that has the $\n$-UL property has the UL property.
\end{enumerate}
\end{corollary}

Note that there is a significant difference between the behavior of the extensions to our context of the UL property and unconditionality for constant coefficients for general bases or Markushevich bases on one hand, and Schauder bases on the other hand. Similar differences occur when we consider democracy-like properties, as we shall see in the next section.

\section{$\n$-democracy and some democracy-like properties}\label{sectiondemocracy}
In greedy approximation theory,  democracy and several similar  properties are widely used for the characterization of greedy-like bases (see for instance in \cite{DKKT, DKO2015, KT}). Here, we study natural extensions of some of these properties to the general context of sequences with gaps. We begin our study with the extensions of two well-known properties. 

\begin{defi}\label{definitionsuperdemo}
	We say that $\mathcal{B}$ is $\n$-superdemocratic if there exists a positive constant $\C$ such that 
	\begin{eqnarray}\label{demo}
	\Vert \mathbf{1}_{\varepsilon A}\Vert \leq \C\Vert \mathbf{1}_{\varepsilon' B}\Vert,
	\end{eqnarray}
	for all $A,B$ with $\vert A\vert\leq \vert B\vert$, $\vert A\vert,\vert B\vert\in\n$ and $\varepsilon\in\Psi_A, \varepsilon'\in\Psi_B$. The smallest constant verifying \eqref{demo} is denoted by $\Delta_s$ and we say that $\mathcal B$ is $\Delta_s$-$\n$-superdemocratic.
	
	If \eqref{demo} is satisfied for $\varepsilon\equiv\varepsilon'\equiv 1$, we say that $\mathcal B$ is $\Delta_d$-$\n$-democratic, where $\Delta_d$ is again the smallest constant for which the inequality holds. 
If $\n=\mathbb N$, we say that $\B$ is $\Delta_d$-democratic and $\Delta_s$-superdemocratic.
\end{defi}

\begin{remark}
	\rm As in the standard case (\cite{DKKT}), it is immediate that a basis is $\n$-superdemocratic if and only if it is $\n$-democratic and $\n$-unconditional for constant coefficients.
\end{remark}

\begin{remark}\label{remarkalsosuperdemocracy}\rm Note that a straightforward convexity argument gives that basis $\mathcal B$ is $\Delta_s$-$\n$-superdemocratic if and only if $\Delta_s$ is the minimum $\C$ for which 
	\begin{eqnarray*}
	\Vert \mathbf{1}_{\varepsilon A}\Vert \leq \C\Vert \mathbf{1}_{\varepsilon' B}\Vert,
	\end{eqnarray*}
for all $A,B$ with $\vert A\vert= \vert B\vert \in\n$ and all $\varepsilon\in\Psi_A, \varepsilon'\in\Psi_B$. Alternatively, this is equivalent to ask that $|A|\le |B|$ and only that $|B|\in \n$.  
\end{remark}

As in the cases of the $\n$-UL property and $\n$-unconditionality for constant coefficients, a key distinction is whether the sequences have (in either sense) bounded gaps. We begin with the results for Schauder bases. 

\begin{lemma}\label{lemmdemaarbitrarilylargegaps}Let $\n$ be a sequence with arbitrarily large quotient gaps. There is a Banach space $\X$ with a Schauder basis $\B$ that is $\n$-superdemocratic but not democratic. 
\end{lemma}
\begin{proof}
See Example~\ref{examplendemocratic} and Remark~\ref{remarkunconditional}. 
\end{proof}

Before we prove our next proposition, we prove a lemma that will be used throughout the paper. 

\begin{lemma}\label{lemmakey}Let $\X$ be a Banach space, $\n$ a sequence with $l$-bounded quotient gaps, and $A\subset \N$ a finite nonempty set, and $(x_j)_{j\in A}\subset \X$. The following hold: 
\begin{enumerate}[i)]
\item \label{previous} Either 
$$
\max_{E\subset A}\left \Vert \sum_{j\in E}x_j\right\Vert\le (n_1-1) \max_{j\in A} \left \Vert x_j\right \Vert 
$$
or there is $B\subset A$ with $|B|\in \n$ such that 
$$
\max_{E\subset A}\left \Vert \sum_{j\in E}x_j\right\Vert \le l\left \Vert \sum_{j\in B}x_j\right \Vert . 
$$
\item \label{new} Given $(b_j)_{j\in A}$ with $|b_j|\ge 1$ for all $j\in A$, either 
\begin{align*}
\max_{\substack{(a_j)_{j\in A}\subset \mathbb{F} \\|a_j|\le 1\forall j\in A}} \left \Vert \sum_{j\in A}a_jx_j\right\Vert\le 2\kappa (n_1-1) \max_{j\in A} \left \Vert x_j\right \Vert 
\end{align*}
or there is $B\subset A$ with $|B|\in \n$ such that 
$$
\max_{\substack{(a_j)_{j\in A}\subset \F\\|a_j|\le 1\forall j\in A}}\left \Vert \sum_{j\in A}a_j x_j\right \Vert \le 2\kappa l\left \Vert \sum_{j\in B}b_j x_j\right \Vert.
$$
\end{enumerate}
\end{lemma}
\begin{proof}
\ref{previous} 
Define 
$$
\Y:=\left\lbrace \sum_{j\in A}a_jx_j\;\colon a_j\in \R\; \forall j\in A\right\rbrace. 
$$
It is immediate that $\Y$ is a finite dimensional Banach space over $\R$ with the norm inherited from $\X$. Since the norms $\|\cdot\|_{\X}$ and $\|\cdot\|_{\Y}$ are the same for elements of $\Y$, we may work in $\Y$ to establish our result. We will denote the norm by $\|\cdot\|$ as in the statement. \\
Pick $D\subset A$ so that
$$
\left \Vert \sum_{j\in D}x_j\right \Vert \ge \left \Vert \sum_{j\in E}x_j\right \Vert \qquad\forall E\subset A. 
$$
If $|D|< n_1$, then 
\begin{align}
\left \Vert \sum_{j\in E}x_j\right \Vert\le \left \Vert \sum_{j\in D}x_j\right \Vert \le (n_1-1)\max_{j\in A}\left \Vert x_j\right \Vert\qquad \forall E\subset A.\label{smallcardinal7}
\end{align}
On the other hand, if $|D|\ge n_1$, set $$k_0:=\max_{k\in \N}\{n_k\le |D|\},$$ and choose $y^*\in \Y^*$ with $\left \Vert y^*\right \Vert =1$ so that 
$$
y^{*}\left(\sum_{j\in D}x_j\right)=\left \Vert \sum_{j\in D}x_j\right \Vert.
$$
Note that, if $\emptyset \subsetneq E\subset D$, 
$$
\left\vert \sum_{j\in E}y^*(x_j)\right\vert \le  \left \Vert \sum_{j\in E}x_j\right \Vert \le \left \Vert \sum_{j\in D}x_j\right \Vert =\sum_{j\in D}y^*(x_j).
$$
Hence, 
$$
y^*(x_j)\ge 0\qquad\forall j\in D. 
$$
Choose $B\subset D$ with $|B|=n_{k_0}$ so that 
$$
y^*(x_j)\ge y^*(x_i)\qquad\forall j\in B\forall i\in D\setminus B. 
$$
Given that $|D|\le l|B|$, for each $E\subset A$ we have
\begin{align*}
\left \Vert \sum_{j\in E}x_j\right \Vert\le \left \Vert \sum_{j\in D}x_j\right \Vert  =\sum_{j\in D}y^*(x_j)\le l  \sum_{j\in B}y^*(x_j)\le l \left \Vert \sum_{j\in B}x_j\right \Vert.
\end{align*}
The proof of \ref{previous} is completed combining the above inequality with \eqref{smallcardinal7}. \\
\ref{new} For each $j\in A$, let $y_j:=b_j x_j$, and choose $D \subset A$ so that
$$
\left \Vert \sum_{j\in D}y_j\right \Vert \ge \left \Vert \sum_{j\in E}y_j\right \Vert \qquad\forall E\subset A. 
$$
Using convexity we obtain
\begin{align*}
\max_{\substack{(a_j)_{j\in A}\subset \mathbb{F} \\|a_j|\le 1\forall j\in A}} \left\Vert \sum_{j\in A}a_j x_j\right\Vert\le &\max_{\substack{(a_j)_{j\in A}\subset \mathbb{F} \\|a_j|\le 1\forall j\in A}} \left\Vert \sum_{j\in A}a_j y_j\right\Vert\le \max_{\substack{(a_j)_{j\in A}\subset \mathbb{F} \\|a_j|\le 1\forall j\in A}}\left(\left\Vert \sum_{j\in A}\Re(a_j) y_j\right\Vert+\left\Vert \sum_{j\in A}\Im(a_j) y_j\right\Vert\right)\\
\le& \kappa \max_{\substack{\lambda_j\in \{-1,1\}\\ \forall j\in E}}\left\Vert \sum_{j\in A}\lambda_j y_j\right\Vert\le 2\kappa \left \Vert \sum_{j\in D}y_j\right \Vert.
\end{align*}
The proof is completed by an application of \ref{previous} to $(y_j)_{j\in D}$. 
\end{proof}

\begin{proposition}\label{propositionndemdem}Suppose $\mathbf{n}$ has $l$-bounded quotient gaps, and let $\B$ be a  Schauder basis with basis constant $\K$. Then: 
\begin{enumerate}[i)]
\item \label{dem15a} If $\B$ is $\Delta_d$-$\mathbf{n}$-democratic, it is $\C$-democratic with 
$$
\C\le \max\{(n_1-1)\alpha_1\alpha_2, l\K  \Delta_d \}.
$$
\item \label{superdem15a} If $\B$ is $\Delta_s$-$\mathbf{n}$-superdemocratic, it is $\C$-superdemocratic with 
$$
\C\le \max\{(n_1-1)\alpha_1\alpha_2, l\K  \Delta_s \}.
$$
\end{enumerate}
\end{proposition}
\begin{proof}
\ref{dem15a} Fix finite sets $A, B$ with $|A|\le |B|$. If $\|\bff_A\|\le (n_1-1) \alpha_1 $, then \begin{equation}
	\|\bff_A\|\le (n_1-1)\alpha_1\alpha_2\|\bff_B\|.\label{lessthann1demsupdem}
	\end{equation}
Otherwise, by Lemma~\ref{lemmakey} there is $A_0\subset A$ with $|A_0|\in \n$ such that 
$$
\|\bff_{A} \|\le l \|\bff_{A_0}\|. 
$$	
Let $B_0$ be the set consisting in the first $n_{k_0}$ elements of $B$. We have 
\begin{align}
\|\bff_{A_0}\|\le \Delta_d \|\bff_{B_0} \|\le \Delta_d \K\|\bff_B\|.\label{replace}
\end{align}
Thus, 
$$
\|\bff_{A}\|\le l\Delta_d \K \|\bff_{B}\|. 
$$
Combining the above inequality with \eqref{lessthann1demsupdem} we obtain that $\B$ is democratic with constant as in the statement. \\
\ref{superdem15a} This is proved by the same argument as \ref{dem15a}. 
\end{proof}
Note that the Schauder condition in Proposition~\ref{propositionndemdem} can be replaced with unconditionality for constant coefficients.  
\begin{lemma}\label{lemmademdem2}Suppose $\mathbf{n}$ is a sequence with $l$-bounded quotient gaps, and $\B$ is a basis that is $\K_u$-unconditional for constant coefficients. Then: 
\begin{enumerate}[i)]
\item \label{dem15} If $\B$ is $\Delta_d$-$\mathbf{n}$-democratic, then 
\begin{enumerate}[1)]
\item \label{demdem}$\B$ is $\C$-democratic with 
$$
\C\le   \max\{(n_1-1)\alpha_1\alpha_2, l\K_u  \Delta_d \}.
$$
\item \label{demsup}$\B$ is $\mathbf{M}$-superdemocratic with 
$$
\mathbf{M}\le \min\{\max\{(n_1-1)\alpha_1\alpha_2, l\K_u^2  \Delta_d \},2 \kappa \max\{ (n_1-1)\alpha_1\alpha_2, l \K_u\Delta_d\}\}. 
$$
\end{enumerate}
\item \label{superdem15} If $\B$ is $\Delta_s$-$\mathbf{n}$-superdemocratic, it is $\C$-superdemocratic with 
$$
\C\le \max\{(n_1-1)\alpha_1\alpha_2, l\K_u  \Delta_s \}.
$$
\end{enumerate}
\end{lemma}
\begin{proof}
\ref{dem15}\ref{demdem}. This is proved by the same argument as Proposition~\ref{propositionndemdem}, with the only difference that instead of \eqref{replace} we get 
\begin{align*}
\|\bff_{A_0}\|\le \Delta_d \|\bff_{B_0} \|\le\Delta_d \max_{\epsilon\in \{-1,1\}}\|\bff_{B_0} + \epsilon \bff_{B\setminus B_0}\|  \le \Delta_d \K_u\|\bff_B\|.
\end{align*}
\ref{dem15}\ref{demsup}. Fix finite sets $A, B$ with $|A|\le |B|$, $\varepsilon\in\Psi_A$ and $\varepsilon'\in \Psi_B$. If $\|\bff_{\bfe A}\|\le (n_1-1) \alpha_1 $, then 
then \begin{equation*}
	 \|\bff_{\bfe A}\|\le (n_1-1)\alpha_1\alpha_2 \|\bff_{\bfe' B}\|.
	\end{equation*}
Otherwise, by Lemma~\ref{lemmakey}\ref{previous} there is $A_0\subset A$ with $|A_0|\in \n$ such that 
$$
\|\bff_{\bfe A}\|\le l  \|\bff_{\bfe A_0}\|.
$$
Choose $B_0\subset B$ with $|B_0|=|A_0|$. We have
\begin{align*}
\|\bff_{\bfe A}\|\le& l  \|\bff_{\bfe A_0}\|\le l \K_u\|\bff_{A_0}\|\le l \K_u\Delta_d \|\bff_{B_0}\|\le l \K_u\Delta_d\max_{\epsilon\in \{-1.1\}}\|\bff_{B_0}+\bff_{B\setminus B_0}\|\le l\K_u^2\Delta_d\|\bff_{\bfe' B}\|.
\end{align*}
Similarly, if $\|\bff_{\bfe A}\|> 2 \kappa (n_1-1) \alpha_1 $, by Lemma~\ref{lemmakey}\ref{new} there is $A_0\subset A$ with $|A_0|\in \n$ such that
$$
\|\bff_{\bfe A}\|\le 2 \kappa l  \|\bff_{ A_0}\|.
$$
Thus, choosing $B_0$ as above we obtain 
$$
\|\bff_{\bfe A}\|\le 2\kappa l \|\bff_{A_0}\|\le 2\kappa l \K_u\Delta_d \|\bff_{\bfe' B}\|.
$$
\ref{superdem15}. This is proved in the same manner as \ref{dem15}\ref{demdem}.
\end{proof}

For general bases, we have the following result. 
\begin{proposition}\label{propositionsuperdemboundedgaps}Let $\n$ be a sequence with gaps. Then:
\begin{enumerate}[i)]
\item \label{arbitrarilylargensupdemo} If $\n$ has arbitrarily large additive gaps, there is a Banach space $\X$ with a Markushevich basis $\B$ that is $\n$-superdemocratic but not democratic. 
\item \label{boundednsupdem} If $\n$ has $l$-bounded additive gaps and $\B$ is $\Delta_{sd}$-$\n$-superdemocratic, it is $\C$-superdemocratic, with 
$$\C\le  \max\{n_1\alpha_1\alpha_2,\Delta_{sd}(1+l\alpha_1\alpha_2)+l\alpha_1\alpha_2\}.$$ 
\item \label{boundedndem} If $\n$ has $l$-bounded additive gaps and $\B$ is $\Delta_{d}$-$\n$-democratic, it is $\C$-democratic, with 
$$\C\le  \max\{(n_1-1)\alpha_1\alpha_2,\Delta_{d}(1+l\alpha_1\alpha_2)+l\alpha_1\alpha_2\}.$$ 
\end{enumerate}
\end{proposition}
\begin{proof}
\ref{arbitrarilylargensupdemo}. See Example~\ref{examplendemocracylargelineargaps}.\\
\ref{boundednsupdem}. Choose $A,B, \varepsilon, \varepsilon'$ as in Definition~\ref{definitionsuperdemo}. If $|A|\le n_1-1$, we have
$$
\|\bff_{\bfe A}\|\le \alpha_1\alpha_2 (n_1-1)\|\bff_{\bfe' B}\|.
$$
Otherwise, there are $k_0\in \N$ such that $n_{k_0}\le |A|\le n_{k_{0}+1}$ and $k_1\ge k_0$ such that $n_{k_1}\le |B|\le n_{k_{1}+1}$. Choose $A_1\subset A$ and $B_1\subset B$ with $|A_1|=n_{k_0}$ and $|B_1|=n_{k_1}$. We have 
\begin{eqnarray*}
\|\bff_{\bfe A}\|&\le& \|\bff_{\bfe A_1}\|+\|\bff_{\bfe A\setminus A_1}\|\le \Delta_{sd}\|\bff_{\bfe' B_1}\|+l\alpha_1\alpha_2 \|\bff_{\bfe' B}\|\\
&\le& \Delta_{sd}\|\bff_{\bfe' B}\|+\Delta_{sd}\|\bff_{\bfe' B\setminus B_1}\|+l\alpha_1\alpha_2 \|\bff_{\bfe' B}\|\le (\Delta_{sd}(1+l\alpha_1\alpha_2)+l\alpha_1\alpha_2) \|\bff_{\bfe' B}\|.
\end{eqnarray*}

\ref{boundedndem} is proven in the same way as \ref{boundednsupdem}. 
\end{proof}

Next, we consider extensions of two other properties: conservativeness and superconservativeness (see \cite{BBL} and \cite{DKKT}).

\begin{defi}\label{definitionsupercon}
	We say that a basis $\mathcal{B}$ is $\n$-superconservative if there exists a positive constant $\C$ such that 
	\begin{eqnarray}\label{cons}
	\Vert \mathbf{1}_{\varepsilon A}\Vert \leq \C\Vert \mathbf{1}_{\varepsilon' B}\Vert,
	\end{eqnarray}
	for all $A,B\subset\N$ with $\vert A\vert\leq \vert B\vert$, $\vert A\vert,\vert B\vert\in\n$, $A<B$, and $\varepsilon\in\Psi_A, \varepsilon'\in\Psi_B$. The smallest constant verifying \eqref{demo} is denoted by $\Delta_{sc}$ and we say that $\mathcal B$ is $\Delta_{sc}$-$\n$-superconservative.
	
	If \eqref{cons} is satisfied for $\varepsilon\equiv\varepsilon'\equiv 1$, we say that $\mathcal B$ is $\Delta_c$-$\n$-conservative, where $\Delta_c$ is the smallest constant for which the inequality holds. 
	
	For $n=\N$, we say that $\B$ is $\Delta_{sc}$-superconservative and $\Delta_c$-conservative.
\end{defi}

The extensions of these two properties to the context of sequences with gaps behave like the extensions of democracy and superdemocracy, in the sense shown in the following results, counterparts of the ones proven above. 

\begin{lemma}\label{lemmanconcon}Suppose $\n$ has $l$-bounded quotient gaps, and let $\B$ be a  Schauder basis with basis constant $\K$. Then: 
\begin{itemize}
\item If $\B$ is $\Delta_{c}$-$\mathbf{n}$-conservative, it is conservative with  constant no greater than $\max\{(n_1-1)\alpha_1\alpha_2, l \Delta_c\K\}$.
\item If $\B$ is $\Delta_{sc}$-$\mathbf{n}$ superconservative, it is superconservative with constant no greater than $\max\{(n_1-1)\alpha_1\alpha_2, l\Delta_{sc}\K.\}$.
\end{itemize}
\end{lemma}
\begin{proof}
This is proved in the same manner as Proposition~\ref{propositionndemdem}.
\end{proof}

\begin{lemma}\label{lemmconsarbitrarilylargegaps}Let $\n$ be a sequence with arbitrarily large quotient gaps. There is a Banach space $\X$ with a Schauder basis $\B$ that is $\n$-superconservative but not conservative. 
\end{lemma}
\begin{proof}
See Example~\ref{examplendemocratic}.
\end{proof}

\begin{lemma}\label{lemmaconcon2S} Suppose $\mathbf{n}$ is a sequence with $l$-bounded quotient gaps, and $\B$ is a basis that is $\K_u$-unconditional for constant coefficients. Then: 
\begin{enumerate}[i)]
\item \label{con15} If $\B$ is $\Delta_c$-$\mathbf{n}$-conservative, then 
\begin{enumerate}[1)]
\item \label{concon}$\B$ is $\C$-conservative with 
$$
\C\le \max\{(n_1-1)\alpha_1\alpha_2, l\K_u  \Delta_c \}
$$
\item \label{consup}$\B$ is $\mathbf{M}$-superconservative with 
$$
\mathbf{M}\le \max\{(n_1-1)\alpha_1\alpha_2, l\K_u^2  \Delta_c \}
$$
\end{enumerate}
\item \label{supercon15} If $\B$ is $\Delta_{sc}$-$\mathbf{n}$-superconservative, it is $\C$-superconservative with 
$$
\C\le \max\{(n_1-1)\alpha_1\alpha_2, l\K_u  \Delta_{sc} \}.
$$
\end{enumerate}
\end{lemma}
\begin{proof}
This Lemma is proved by the same arguments as Lemma~\ref{lemmademdem2}, with only straightforward modifications. 
\end{proof}

\begin{lemma}\label{lemmasuperconsvoundedgaps}Let $\n$ be a sequence with gaps. Then 
\begin{enumerate}[i)]
\item If $\n$ has arbitrarily large additive gaps, there is a Banach space $\X$ with a Markushevich basis $\B$ that is $\n$-superconservative but not conservative. 
\item  If $\n$ has $l$-bounded additive gaps and $\B$ is $\Delta{sc}$-$\n$-superconservative, it is $\C$-superconservative, with 
$$\C\le  \max\{(n_1-1)\alpha_1\alpha_2,\Delta_{sc}(1+l\alpha_1\alpha_2)+l\alpha_1\alpha_2\}.$$ 
\item  If $\n$ has $l$-bounded additive gaps and $\B$ is $\Delta{c}$-$\n$-conservative, it is $\C$-conservative, with 
$$\C\le  \max\{(n_1-1)\alpha_1\alpha_2,\Delta_{c}(1+l\alpha_1\alpha_2)+l\alpha_1\alpha_2\}.$$ 
\end{enumerate}
\end{lemma}

\begin{proof}
This is proved in the same manner as Proposition~\ref{propositionsuperdemboundedgaps}. 
\end{proof}

\section{$\n$-symmetry and $\n$-quasi-greediness for largest coefficients}\label{sectionnsymmetry}

In this section, we extend to the context of sequences with gaps the notions of quasi-greediness for largest coefficients and symmetry for largest coefficients. We also study an extension of suppression unconditionality for constant coefficients. We begin with the first of these properties, introduced in \cite{AABW}.

\begin{definition}\label{definitionnQGlc} We say that $\mathcal B$ is $\mathbf{n}$-quasi-greedy for largest coefficients if there exists a positive constant $\C$ such that
	\begin{eqnarray}\label{ql}
	\|\bff_{\bfe A}\|\le \C\|\bff_{\bfe A}+x\|
	\end{eqnarray}
	for every $A\subset \N$ with $|A|\in \mathbf{n}$, $\bfe \in \Psi_A$, and all $x\in \mathbb{X}$ such that $\supp{(x)}\cap A=\emptyset$ and $|\be_i^*(x)|\le 1$ for all $i\in \N$. The smallest constant verifying \eqref{ql} is denoted by $\C_{ql}$ and we say that $\B$ is $\C_{ql}$-$\n$-quasi-greedy for largest coefficients. When $\mathbf{n}=\N$, $\B$ is $\C_{ql}$-quasi-greedy for largest coefficients.  
\end{definition}
It is immediate that if $\B$ is $\C_{q,t}$-$t$-$\mathbf{n}$-quasi-greedy, it is also $\C_{ql}$-$\mathbf{n}$-quasi-greedy for largest coefficients with $\C_{ql}\leq \C_{q,t}$.\\
 Note that it is enough to take $x$ a finite linear combination of some of the $\be_j$'s in Definition~\ref{definitionnQGlc}. More precisely, we have the following elementary characterization. 

\begin{lemma}\label{equiv0}
	A basis $\mathcal B$ is $\n$-quasi-greedy for largest coefficients if and only if there exists a positive constant $\mathbf{L}$ such that
	\begin{eqnarray}\label{qglc1}
	\Vert \mathbf{1}_{\varepsilon A}\Vert \le  \mathbf{L}\Vert x+\mathbf{1}_{\varepsilon A}\Vert,
	\end{eqnarray}
	for every $A\subset \N$ with $|A|\in \mathbf{n}$, $\bfe \in \Psi_A$, and all $x\in [\be_j: j\in \N]$ such that $\supp{(x)}\cap A=\emptyset$ and $|\be_j^*(x)|\le 1$ for all $j\in \N$. Moreover, if \eqref{qglc1} holds, then $\C_{ql}\le \mathbf{L}$. 
\end{lemma}
\begin{proof}
Clearly we only need to show that if \eqref{qglc1} holds, then it also holds for $x\in \X\setminus [\be_j: j\in \N]$, that is for $x$ which is not a finite linear combination of some of the $\be_j$'s. Given such $x$, there is a sequence $(x_k)_{k \in \N}\subset [\be_j: j\in \N]$ such that 
$$
x_k\xrightarrow[k\to \infty]{} x. 
$$
For each $k$, let $y_k:=x_k-P_A(x_k)$. Since $\be_j^*(x)=0$ for all $j\in A$ and $A$ is finite, we have $$
y_k\xrightarrow[k\to \infty]{} x,$$ so $$\|y_k\|_{\infty}\xrightarrow[k\to \infty]{}\|x\|_{\infty}.$$ Hence, if $\|x\|_{\infty}<1$, there is $k_0\in \N$ such that $\|y_k\|_{\infty}\le 1$ for all $k\ge k_0$, so 
\begin{align*}
\|\bff_{\bfe A}\|\le&  \mathbf{L} \|\bff_{\bfe A}+y_{k+k_0}\|\xrightarrow[k\to \infty]{} \mathbf{L} \|\bff_{\bfe A}+x\|.
\end{align*}
On the other hand, if $\|x\|_{\infty}=1$, define 
\begin{equation}z_k:=
\begin{cases}
\|y_k\|_{\infty}^{-1}y_k & \text{ if } y_k\not=0;\\
0 & \text{ otherwise}.  
\end{cases}
\end{equation}
Since 
$$
z_k\xrightarrow[k\to \infty]{} x
$$
and $\|z_k\|_{\infty}\le 1$ for all $k\in \N$, the proof is completed by the same argument used in the case $\|x\|_{\infty}<1$.
\end{proof}

If a basis is quasi-greedy for largest coefficients, it is unconditional for constant coefficients (this follows for example from the proof of \cite[Proposition 3]{Wo}, or from \cite[Remark 3.4]{BBG}). Hence, Example~\ref{examplendemocratic} shows that, for $\n$ with arbitrarily large quotient gaps, $\n$-quasi-greediness for largest coefficients is not equivalent to its regular counterpart. On the other hand, the following proposition shows that equivalence holds in the remaining cases.

\begin{proposition}\label{propositioncglcbqg} Suppose $\mathbf{n}$ has $l$-bounded quotient gaps, and $\B$ is $\n$-$\C_{ql}$-quasi-greedy for largest coefficients. Then $\B$ is $\C$-quasi-greedy for largest coefficients with 
$$
\C\le \max\{(n_1-1)\alpha_1\alpha_2, l\C_{ql} \}
$$
\end{proposition}
\begin{proof}
Fix a finite set $A\subset \N$ with $0<|A|\not\in \mathbf{n}$, and $x$, $\bfe \in \Psi_A$ as in Definition~\ref{definitionnQGlc}. If $\|\bff_{\bfe A}\|\le (n_1-1)\alpha_1$, then 
\begin{align*}
\|\bff_{\bfe A}\|\le (n_1-1)\alpha_1\alpha_2 \|\bff_{\bfe A}+x\|. 
\end{align*}
Otherwise, by Lemma~\ref{lemmakey}, there is $B\subset A$ with $|B|\in \n$ such that $\|\bff_{\bfe A}\|\le l \|\bff_{\bfe B}\|$. Hence, 
$$
\|\bff_{\bfe A}\|\le l \|\bff_{\bfe B}\|\le l \C_{ql} \|\bff_{\bfe B}+\bff_{\bfe A\setminus B}+x\|=  l \C_{ql}\|\bff_{\bfe A}+x\|, 
$$
and the proof is complete. 
\end{proof}
Next, we consider an extension of suppression unconditionality for constant coefficients, a property studied in \cite{ AABW,BBL,BBG}, among others.  This property is equivalent to unconditionality for constant coefficients (see \cite[Remark 3.4]{BBG}) but, as we shall see, their extensions to our context behave differently and are not in general equivalent.

\begin{definition}\label{snuccc} We say that $\B$ is $\n$-suppression unconditional for constant coefficients if there is $\C>0$ such that 
\begin{eqnarray*}
	\|\bff_{\bfe A}\|\le \C\|\bff_{\bfe' B}\|
\end{eqnarray*}
	for all $A\subset B\subset \N$ with $|A|\in\n$ and all $\bfe' \in \Psi_B$. The smallest constant verifying the above inequality is denoted by $\mathbf{K}_{su}$ and we say that $\mathcal B$ is $\K_{suc}$-$\n$-suppression unconditional for constant coefficients. If $\n=\mathbb N$, we say that $\mathcal B$ is $\mathbf{K}_{su}$-suppression unconditional for constant coefficients.
\end{definition}
It is immediate from the definition that if $\B$ is $\C_{ql}$-$\n$-quasi-greedy for largest coefficients, it is $\K_{suc}$-$\n$-suppression unconditional for constant coefficients with $\K_{suc}\le \C_{ql}$. \\
Unlike $\n$-unconditionality for constant coefficients  (see Propositions~\ref{propositionucccboundedgaps} and~\ref{propositionboundedgapsnULnoUl}), for sequences with bounded quotient gaps $\n$-suppression unconditionality for constant coefficients is equivalent to its regular counterpart. 

\begin{proposition}\label{propositionnsuccsucc}Suppose $\n$ has $l$-bounded quotient gaps, and  $\B$ is $\K_{suc}$-$\n$-suppression unconditional for constant coefficients.  Then $\B$ is $\C$-suppression unconditional for constant coefficients, with 
$$
\C\le \max\{(n_1-1)\alpha_1\alpha_2, l \K_{suc} \}.
$$
\end{proposition}
\begin{proof}
This is proven by a simpler variant of the argument of Proposition~\ref{propositioncglcbqg}, taking $x=\bff_{\bfe' E}$ for some finite set $E\subset\N$ and $\varepsilon'\in \Psi_E$. 
\end{proof}

Finally, we extend the property of being symmetric for largest coefficients to the context of sequences with gaps. This property was introduced in \cite{AW} (as Property (A)) and studied in \cite{BOB, BBG, BDKOW,DKOSZ, BBG}.

\begin{defi}\label{definitionnslq}
	We say that $\mathcal{B}$ is $\n$-symmetric for largest coefficients if there exists a positive constant $\C$ such that 
	\begin{eqnarray}\label{sy}
	\Vert x+\mathbf{1}_{\varepsilon A}\Vert \leq \C\Vert x+\mathbf{1}_{\varepsilon' B}\Vert,
	\end{eqnarray}
	for any pair of sets $A,B$ with $\vert A\vert\leq \vert B\vert$, $A\cap B=\emptyset$, $\vert A\vert,\vert B\vert\in\n$, for any $\varepsilon\in\Psi_A, \varepsilon'\in\Psi_B$ and for any $x\in\X$ such that $\vert\be_i^*(x)\vert\leq 1\, \forall i\in\mathbb N$ and $\supp(x)\cap (A\cup B)=\emptyset$. The smallest constant verifying \eqref{sy} is denoted by $\Delta$ and we say that $\mathcal B$ is $\Delta$-$\n$-symmetric for largest coefficients. If $\n=\N$, we say that $\mathcal B$ is $\Delta$-symmetric for largest coefficients.
\end{defi}

Note that our definition is equivalent to only requiring that $|A|= |B|\in \n$ instead of $|A|\leq|B|\in\n$, and $x\in [\be_j: j\in \N]$. The following lemma proves these facts.

\begin{lemma}\label{equiv}
	A basis $\mathcal B$ is $\n$-symmetric for largest coefficients if and only if there exists a positive constant $\mathbf{L}$ such that
	\begin{eqnarray}\label{sy1}
	\Vert x+\mathbf{1}_{\varepsilon A}\Vert \le  \mathbf{L}\Vert x+\mathbf{1}_{\varepsilon' B}\Vert,
	\end{eqnarray}
	for any pair of sets $A,B$ with $\vert A\vert=  \vert B\vert$, $A\cap B=\emptyset$, $\vert B\vert\in\n$, for any $\varepsilon\in\Psi_A, \varepsilon'\in\Psi_B$ and for any $x\in [\be_j: j\in \N]$ such that $\vert\be_i^*(x)\vert\leq 1\, \forall i\in\mathbb N$ and $\supp(x)\cap (A\cup B)=\emptyset$. Moreover, $\Delta$ is the minimum $\mathbf{L}$ for which \eqref{sy1} holds.
\end{lemma}
\begin{proof}
	Of course, we only have to show that \eqref{sy1} implies $\n$-symmetry for largest coefficients with constant no greater than $\mathbf{L}$. Let $x, A, B, \varepsilon, \varepsilon'$ be as in Definition~\ref{definitionnslq}, with the additional condition that $x\in [\be_j: j\in \N]$. If $|A|=|B|\in \mathbf{n}$, there is nothing to prove. 
	 Else, choose a set $C>\supp(x)\cup A\cup B$ such that $\vert A\vert+\vert C\vert=\vert B\vert\in\n$. We have
	$$\Vert x+\mathbf{1}_{\varepsilon A}\Vert \leq \frac{1}{2}\left(\Vert x+\mathbf{1}_{\varepsilon A}+\mathbf 1_C\Vert+\Vert x+\mathbf{1}_{\varepsilon A}-\mathbf{1}_C\Vert  \right)\leq \mathbf{L}\Vert x+\mathbf{1}_{\varepsilon' B}\Vert. $$
To prove the result for $x\not \in [\be_j: j\in \N]$, apply the argument of Lemma~\ref{equiv0}. 
	\end{proof}
	
\begin{remark}\rm Note that for Markushevich bases, $x\in [\be_j:j\in \N]$ if and only if $x$ has finite support, so for such bases Lemmas~\ref{equiv0} and~\ref{equiv} can be proved using \cite[Lemma 3.2]{BDKOW} (a result that can also be extended to bases that are not total, with only a slight modification of the proof). 
\end{remark}

Next, we study the relation between $\n$-symmetry for largest coefficients and $\n$-superdemocracy. 

\begin{lemma}\label{rem1}\rm
Let $\B$ be a basis. If $\B$ is $\Delta$-$\n$-symmetric for largest coefficients, it is $\Delta_s$-$\n$-superdemocratic with $\Delta_s\leq \Delta^2$.
\end{lemma}
\begin{proof}
Consider two sets $A, B$ with cardinality in $\n$ and $\vert A\vert\leq\vert B\vert$, and a set $C> A\cup B$ such that $\vert C\vert=\vert A\vert$. Then,
	\begin{eqnarray}\label{sym}
	\frac{\Vert\mathbf{1}_{\varepsilon A}\Vert}{\Vert\mathbf{1}_{\varepsilon' B}\Vert}=\frac{\Vert\mathbf{1}_{\varepsilon A}\Vert}{\Vert\mathbf{1}_{C}\Vert}\frac{\Vert\mathbf{1}_{C}\Vert}{\Vert\mathbf{1}_{\varepsilon' B}\Vert}\leq \Delta^2.
	\end{eqnarray}
\end{proof}

In the case $\n=\mathbb N$, it is known that if $\mathcal B$ is $\Delta$-symmetric for largest coefficients, then it is $\Delta_s$-superdemocratic with $\Delta_s \leq 2\kappa\Delta$ (\cite[Proposition 1.1]{BBG}). In Lemma~\ref{rem1}, for a general sequence $\n$, we have shown that if $\B$ is $\Delta$-$\n$-symmetric for largest coefficients, it is $\Delta_s$-$\n$-superdemocratic with $\Delta_s\leq\Delta^2$.  This suggests the question of  whether the latter estimate can be improved in the sense that $\Delta_s\lesssim \Delta$. Our next result shows that this is not possible. In fact, it is not even possible to obtain $\Delta_s \lesssim \Delta^p$ for any $1<p<2$.

\begin{proposition}\label{propositionnslcnsdem} Let $0<\delta<1$ and $M>1$. There is a sequence $\n$ and a Banach space $\X$ with a Schauder basis $\B$ that is $\Delta$-$\n$-symmetric for largest coefficients and $\Delta_{s}$-$\n$-superdemocratic 
	with  
	\begin{equation}
	\Delta>M\qquad \text{and}\qquad\Delta_{s}\ge \Delta^{2-\delta}.\nonumber
	\end{equation}
\end{proposition}
\begin{proof}
Fix $0<\epsilon<1<q<p $ so that the following hold: 
\begin{eqnarray}
1-\frac{1}{q}&\le& \frac{1}{q}-\frac{1}{p+\epsilon},\label{cond1}\\
1-\frac{1}{p}&\ge & \left(2-\delta\right)\left(\frac{1}{q}-\frac{1}{p+\epsilon}\right).\label{cond3}
\end{eqnarray}
For example, one can take $q=\frac{8}{5}$ and $p=4-\epsilon$ for a sufficiently small $\epsilon$. Now choose $m\in \N$ an even number sufficiently large so that 
	\begin{equation}
	m^{\frac{1}{q}-\frac{1}{p+\epsilon}}>2+2^{\frac{1}{p}}m^{\frac{1}{q}-\frac{1}{p}}\qquad  \text{and}\qquad m^{1-\frac{1}{p}}>M^2. \label{largem}
	\end{equation}
	Define $\X$ as the completion of $\mathtt{c}_{00}$ with the norm 
	$$
	\left|\left|(a_i)_i\right|\right|:= \max{\left\lbrace \left|\sum\limits_{i=1}^{m}a_i\right|,\left(\sum\limits_{i=1}^{m}\left|a_i\right|^{p}\right)^{\frac{1}{p}}, \left(\sum\limits_{i=m+1}^{\infty}\left|a_i\right|^{q}\right)^{\frac{1}{q}}\right\rbrace},
	$$ 
	and let $\n$ be the sequence $\{m\}\cup \N_{>m^q+m}$, and $B_m:=\{1,\dots,m\}$. \\
	As the norm $\left|\left|\cdot \right|\right|$, when restricted to $\left(\be_i\right)_{i\ge m+1}$, coincides with the usual norm on $\ell_{q}$,  it follows easily that the unit vector basis $\B=\left(\be_i\right)_{i\in \N}$ is a symmetric basis for $\X$, and thus it is symmetric for largest coefficients. Hence, in particular there are constants $\Delta>0$ and $\Delta_s>0$ such that $\B$ is $\Delta$-$\n$-symmetric for largest coefficients and $\Delta_s$-$\n$-superdemocratic. \\
	To estimate $\Delta$, by Lemma~\ref{equiv} it is enough to consider sets $A,B\subset \N$ with $|A|=|B|\in \n$, $\varepsilon\in \Psi_A$, $\varepsilon'\in \Psi_B$, and $x\in\X$ with finite support such that $\vert\be_i^*(x)\vert\leq 1\, \forall i\in\mathbb N$ and $\supp(x)\cap (A\cup B)=\emptyset$. \\First we consider the case $|A|=|B|> m^{q}+m$. Take $D>A\cup B\cup \supp{(x)}$ with $|D|=m$. By \eqref{cond1} we have
	\begin{equation}
	\|P_{B_m}(x+\bff_{\bfe A})\|\le m\le \frac{m}{\left(|B|-m\right)^{\frac{1}{q}}} \left(\sum\limits_{i>m}\left|\be_i^*(\bff_{\bfe A\setminus B_m})\right|^{q}\right)^{\frac{1}{q}}\le \|x+\bff_{\bfe' B}\|.   \label{proj1}
	\end{equation}
On the other hand, 
	\begin{eqnarray}
	\|P_{B_m^c}(x+\bff_{\bfe A})\|&=& \left(\left(\sum\limits_{i=m+1}^{\infty}\left|\be^*_i(x)\right|^{q}\right)+|A\setminus B_m| \right)^{\frac{1}{q}}\nonumber\\
	&\le& \left(\left(\sum\limits_{i=m+1}^{\infty}\left|\be^*_i(x)\right|^{q}\right)+|B\setminus B_m| +|D| \right)^{\frac{1}{q}}\nonumber\\
	&\le&  \|P_{B_m^c}(x+\bff_{\bfe B})+\bff_{ D}\|\le  \|P_{B_m^c}(x+\bff_{\bfe B})\| +\|\bff_{ D}\|\nonumber\\
	&\le & \|x+\bff_{\bfe B}\| +\|P_{B_m^c}(\bff_{\bfe' B})\|\le 2\|x+\bff_{\bfe' B}\|.\label{proj2}
	\end{eqnarray}
	Combining \eqref{proj1} and \eqref{proj2} we obtain
	\begin{equation}
	\|x+\bff_{\bfe A}\|= \max{\left\lbrace\|P_{B_m}(x+\bff_{\bfe A})\|,\|P_{B_m^c}(x+\bff_{\bfe A})\|\ \right\rbrace}\le 2\|x+\bff_{\bfe' B}\|.\label{bound>m}
	\end{equation}
	Now we consider the case $|A|=|B|=m$. As $|\supp{(P_{B_m}(x+\bff_{\bfe A}))}|\le  |B_m\setminus B|=|B\setminus B_m|$, by \eqref{cond1} we have 
	\begin{eqnarray}
	\|P_{B_m}(x+\bff_{\bfe A})\|&\le& |\supp{(P_{B_m}(x+\bff_{\bfe A}))}|\le |B\setminus B_m|^{\frac{1}{q}-\frac{1}{p+\epsilon}}|B\setminus B_m|^{\frac{1}{q}}\nonumber\\
	&\le& m^{\frac{1}{q}-\frac{1}{p+\epsilon}}\|P_{B_m^c}(\bff_{\bfe' B})\|\le m^{\frac{1}{q}-\frac{1}{p+\epsilon}}\|x+\bff_{\bfe' B}\|,\label{proj1m}
	\end{eqnarray}
	and 
	\begin{eqnarray}
	\|P_{B_m^c}(\bff_{\bfe A})\|&\le& \frac{m^{\frac{1}{q}}}{\max{\left\lbrace |B\setminus B_m|^\frac{1}{q}, |B\cap B_m|^{\frac{1}{p}}\right\rbrace}}\max{\left\lbrace|B\setminus B_m|^\frac{1}{q}, |B\cap B_m|^{\frac{1}{p}}\right\rbrace}\nonumber\\
	&\le & \frac{m^{\frac{1}{q}}}{\left(\frac{m}{2}\right)^{\frac{1}{p}}}\max{\left\lbrace \|P_{B_m^c}(x+\bff_{\bfe' B})\|, \|P_{B_m}(x+\bff_{\bfe' B})\|  \right\rbrace}\nonumber\\
	&=&2^{\frac{1}{p}}m^{\frac{1}{q}-\frac{1}{p}}\|x+\bff_{\bfe' B}\|.\label{lastone}
	\end{eqnarray}
	As $\|P_{B_m^c}(x)\|\le \|x+\bff_{\bfe' B}\|$, from \eqref{proj1m}, \eqref{lastone} and the triangle inequality we obtain
	\begin{eqnarray}
	\|x+\bff_{\bfe A}\|&=&\max{\left\lbrace\|P_{B_m}(x+\bff_{\bfe A})\|,\|P_{B_m}^c(x+\bff_{\bfe A})\|\ \right\rbrace}\nonumber\\
	&\le & \max{\left\lbrace m^{\frac{1}{q}-\frac{1}{p}+\epsilon}, 1+ 2^{\frac{1}{p}}m^{\frac{1}{q}-\frac{1}{p}}\right\rbrace}\|x+\bff_{\bfe' B}\|\nonumber\\
	&=&m^{\frac{1}{q}-\frac{1}{p+\epsilon}}\|x+\bff_{\bfe' B}\|,\label{bound=m}
	\end{eqnarray}
	where we used \eqref{largem} for the last estimate. From \eqref{bound>m} and\eqref{bound=m}, using \eqref{largem} we deduce that
	\begin{equation}
	\Delta\le m^{\frac{1}{q}-\frac{1}{p+\epsilon}}.\label{boundfornpslc}
	\end{equation}
	Now let $\varepsilon'\in \Psi_{B_m}$ be any sequence of alternating signs. As $m$ is even, we have
	\begin{equation*}
	\sum\limits_{i=1}^{m}\be_i^*(\bff_{\bfe' B_m})=0. 
	\end{equation*}
	Thus,
	\begin{equation*}
	\|\bff_{\bfe' B_m}\|=m^{\frac{1}{p}}. 
	\end{equation*}
	Since $\|\bff_{B_m}\|=m$, we conclude (using \eqref{cond3} and \eqref{boundfornpslc}) that 
	\begin{equation}
	\Delta_{s}\ge m^{1-\frac{1}{p}}\ge \left(m^{\frac{1}{q}-\frac{1}{p+\epsilon}}\right)^{2-\delta}\ge \Delta^{2-\delta}. \nonumber
	\end{equation}
	Finally, from this result and \eqref{largem}, by Lemma~\ref{rem1} we get 
	\begin{equation}
	\Delta\ge \Delta_s^{\frac{1}{2}}\ge \left( m^{1-\frac{1}{p}}\right)^{\frac{1}{2}}>M.\nonumber
	\end{equation}
\end{proof}

Our next result shows that $\n$ symmetry for largest coefficients can be characterized in terms of $\n$-superdemocracy and $\n$-quasi-greediness for largest coefficients (see \cite{AABW}).

\begin{proposition}\label{propositionequivnslcnsdnqglc}
	A basis $\mathcal B$ is $\n$-symmetric for largest coefficients if and only if $\mathcal B$ is $\n$-superdemocratic and $\n$-quasi-greedy for largest coefficients. Moreover,
	$$\C_{ql}\leq 1+\Delta,\;\; \Delta\leq 1+\C_{ql}(1+\Delta_s).$$
\end{proposition}
\begin{proof}
	To show that $\n$-superdemocracy and $\n$-quasi-greediness together imply $\n$-symmetry for largest coefficients, just follow the proof of \cite[Proposition 4.3]{AABW}. Assume now that $\B$ is $\n$-symmetric for largest coefficients. By Lemma~\ref{rem1}, $\B$ is $\n$-superdemocratic. Given $x\in [\be_j:j\in \N]$, $\vert A\vert\in\n$ with $A\cap\supp(x)=\emptyset$ and $\bfe \in \Psi_A$, choose $C>\supp(x)\cup A$ so that $\vert C\vert=\vert A\vert$. We have
	\begin{eqnarray}\label{pr1}
	\nonumber	\Vert \mathbf{1}_{\varepsilon A}\Vert&\leq& \Vert x+\mathbf{1}_{\varepsilon A}\Vert+\Vert x\Vert\\
	\nonumber&\leq&\Vert x+\mathbf{1}_{\varepsilon A}\Vert+\frac{1}{2}(\Vert x+\mathbf{1}_C\Vert\ +\Vert x-\mathbf{1}_C\Vert)\\
	&\leq& \Vert x+\mathbf{1}_{\varepsilon A}\Vert+\Delta\Vert x+\mathbf{1}_{\varepsilon A}\Vert.
	\end{eqnarray}
	By \eqref{pr1} and Lemma~\ref{equiv0}, $\mathcal B$ is $\C_{ql}$-$\n$-quasi-greedy for largest coefficients with $\C_{ql}\leq 1+\Delta$.
\end{proof}
Our next two results characterize the sequences $\n$ for which $\n$-symmetry for largest coefficients is equivalent to symmetry for largest coefficients. 
\begin{proposition}\label{lemmaarbitrarilylargegaps}Let $\n$ be a sequence with arbitrarily large quotient gaps. There is a Banach space $\X$ with a Schauder basis $\B$ that is $\n$-symmetric for largest coefficients but not democratic. 
\end{proposition}
\begin{proof}
See Example~\ref{examplendemocratic} and Remark~\ref{remarkunconditional}. 
\end{proof}

%
%
%

%

\begin{theorem}\label{theoremnsymlc}
	Let $\mathcal B$ be a basis and assume that $\n$ has $l$-bounded quotient gaps. If $\mathcal B$ is $\Delta$-$\n$-symmetric for largest coefficients, then $\B$ is $\mathbf{C}$-symmetric for largest coefficients with $\mathbf{C}\leq\max\lbrace 1+2(n_1-1)\alpha_1\alpha_2,\, 1+2\Delta^2(1+l)\rbrace$.
\end{theorem}
\begin{proof}
	
	To show that $\mathcal B$ is symmetric for largest coefficients we use Lemma \ref{equiv}: Take $x\in [\be_j: j\in \N]$ so that $\max_{j\in \N}\vert \be_j^*(x)\vert \leq 1$, and two finite sets $A,B\subset \N$ so that $A\cap B=\emptyset$, $\vert A\vert= \vert B\vert$, and $\supp(x)\cap (A\cup B)=\emptyset$.

	Assume first that there exists $i\in\mathbb{N}$ such that $n_i \leq m\leq n_{i+1}$ with $n_i, n_{i+1}\in\n$. Then, we can decompose $A=A_0\cup A_1$ and $B=B_0\cup B_1$ with $\vert A_0\vert=\vert B_0\vert= n_i\in\n$. Thus,
	\begin{eqnarray}\label{one}
	\Vert x+\mathbf{1}_{\varepsilon A}\Vert \leq \Vert x+\mathbf{1}_{\varepsilon' B}\Vert + \Vert\mathbf{1}_{\e A_0}\Vert+\Vert\mathbf{1}_{\e A_1}\Vert +\Vert\mathbf{1}_{\e'B_0}\Vert+\Vert\mathbf{1}_{\e' B_1}\Vert.
	\end{eqnarray} 
	
	Take $C>\supp(x)\cup A\cup B$ such that $\vert C\vert=\vert A_0\vert$. Hence,
	\begin{eqnarray}\label{two}
	\nonumber\Vert\one_{\e A_0}\Vert&\leq&\Delta\Vert\one_C\Vert\leq\frac{\Delta}{2}(\Vert x+\one_{\e'B_1}+\one_C\Vert+\Vert x+\one_{\e'B_1}-\one_C\Vert)\\
	\nonumber&\leq& \Delta\max\lbrace\Vert x+\one_{\e'B_1}+\one_C\Vert,\Vert x+\one_{\e' B_1}-\one_C\Vert\rbrace\\
	&\leq& \Delta^2\Vert x+\one_{\e' B_1}+\one_{\e'B_0}\Vert=\Delta^2\Vert x+\one_{\e' B}\Vert.
	\end{eqnarray}
	Thus,  the same argument for \eqref{two} can be used to estimate $\Vert\one_{\e' B_0}\Vert$, and we obtain that
	\begin{eqnarray}\label{three}
	\max\lbrace\Vert \one_{\e A_0}\Vert,\Vert\one_{\e' B_0}\Vert\rbrace\leq \Delta^2 \Vert x+\mathbf{1}_{\e' B}\Vert.
	\end{eqnarray}
	
	To estimate $\Vert\one_{\e A_1}\Vert$, take now a set $F>\supp(x)\cup A\cup B\cup C$ such that $\vert F\vert+\vert A_1\vert=l n_i$, and write
	$$\one_{\e A_1}\pm\one_F=\sum_{j=1}^l \one_{\eta T_j},$$
	where $T_k\cap T_i=\emptyset$ for $i\neq k$, $\vert T_j\vert=n_i$ for all $j=1,...,l$ and $\eta$ the corresponding sign. Hence, since $\Vert\one_{\eta T_j}\Vert\leq \Delta\Vert\one_{C}\Vert$ for all $j=1,...,l$,
	\begin{eqnarray}\label{four}
	\nonumber\Vert\one_{\e A_1}\Vert&\leq& \frac{1}{2}(\Vert \one_{\e A_1}+\one_F\Vert+\Vert\one_{\e A_1}-\one_F\Vert)\\
	\nonumber&\leq&\max\lbrace\Vert\one_{\e A_1}+\one_F\Vert, \Vert\one_{\e A_1}-\one_F\Vert\rbrace\\
	&\leq& l\Delta\Vert \one_C\Vert\stackrel{\eqref{two}}{\leq}l\Delta^2\Vert x+\one_{\e' B}\Vert.
	\end{eqnarray}
	Applying \eqref{four} to estimate $\Vert\one_{\e' B_1}\Vert$, we obtain
	\begin{eqnarray}\label{five}
	\max\lbrace\Vert\one_{\e A_1}\Vert, \Vert \one_{\e' B_1}\Vert\rbrace \leq \Delta^2 \Vert x+\one_{\e' B}\Vert.
	\end{eqnarray}
	
	Thus, applying \eqref{three} and \eqref{five} in \eqref{one}, 
	$$\Vert x+\one_{\e A}\Vert \leq \mathbf (1+2\Delta^2 +
	2l \Delta^2)\Vert x+\mathbf{1}_{\e' B}\Vert,$$
	for sets $A$ and $B$ with cardinality equal to or greater than $n_1$. Assume now $\vert A\vert<n_1$. In that case,
	\begin{eqnarray*}
		\Vert x+\mathbf{1}_{\varepsilon A}\Vert&\leq& \Vert x+\mathbf{1}_{\varepsilon' B}\Vert +\Vert \mathbf{1}_{\varepsilon A}\Vert+\Vert\mathbf{1}_{\varepsilon' B}\Vert\\
		&\leq&\Vert x+\mathbf{1}_{\varepsilon' B}\Vert+2\alpha_1 (n_1-1)\\
		&\leq& (1+2(n_1-1)\alpha_1\alpha_2)\Vert x+\mathbf{1}_{\varepsilon 'B}\Vert.
	\end{eqnarray*}
	
	Thus, the basis is $\mathbf{C}$-symmetric for largest coefficients with $$\mathbf{C}\leq\max\lbrace 1+2(n_1-1)\alpha_1\alpha_2,\, 1+2\Delta^2(1+l)\rbrace.$$
\end{proof}

To close this section, we use $\n$-democracy and the $\n$-UL property as an alternative to the Schauder condition in \cite[Theorem 5.2]{BB} - where it is proven that if $\n$ has bounded quotient gaps, every $\n$-quasi-greedy Schauder basis is quasi-greedy - and we also obtain symmetry for largest coefficients.

\begin{proposition}\label{propositionndemULboundedgaps}Suppose $\mathbf{n}$ is a sequence with $l$-bounded quotient gaps, and $\B$ is a basis that is $\C_{q,t}$-$t$-$\n$-quasi-greedy and has the $\n$-UL-property with constants $\C_1$ and $\C_2$. Then, the following hold: 
\begin{enumerate}[i)]
\item \label{ndemocratic}If $\B$ is $\Delta_d$-$\mathbf{n}$-democratic, it is $\C$-$t$-quasi-greedy with  
$$
\C\le \max\{(n_1-1)\alpha_1\alpha_2, \C_{q,t}\left(1+\left(l-1\right)\C_1\C_2 \Delta_d\right)\},
$$
and is $\Delta$-symmetric for largest coefficients with 
$$
\Delta\le \max\lbrace 1+2(n_1-1)\alpha_1\alpha_2, 1+2(1+l)(1+\C_{q,t}(1+\C_1\C_2\Delta_d))^2 \rbrace.
$$

\item \label{nsuperdemocratic}
If $\B$ is $\Delta_s$-$\mathbf{n}$-superdemocratic, it is $\C$-$t$-quasi-greedy with  
$$
\C\le \max\{(n_1-1)\alpha_1\alpha_2, \C_{q,t}\left(1+\left(l-1\right)\C_1 \Delta_s\right)\},
$$
and is $\Delta$-symmetric for largest coefficients with 
$$
\Delta\le \max\lbrace 1+2(n_1-1)\alpha_1\alpha_2, 1+2(1+l)(1+\C_{q,t}(1+\Delta_s))^2 \rbrace.
$$
\end{enumerate}
\end{proposition}
\begin{proof}
\ref{ndemocratic} Fix $x\in \X$ and $A$ a $t$-greedy set for $x$ with $\left|A\right|\not\in \n$. If $\left|A\right|<n_1$, then 
$$
\left|\left|P_A(x)\right|\right|\le \sum_{i\in A}\left|\be_i^*(x)\right|\left|\left|\be_i\right|\right|\le \alpha_1\alpha_2 (n_1-1)\|x\|.
$$
If $|A|> n_1$, define
	$$
	k_0:=\max_{k\in \N}\{n_k< |A|\}, 
	$$
and let $\left\lbrace A_i\right\rbrace_{1\le i\le j}$ be a partition of $A$ such that $2\le j\le l$, $A_1$ is an $n_{k_0}$-greedy set for $P_A\left(x\right)$, and $\left|A_i\right|\le  n_{k_0}$ for all $2\le i\le j$. Since $A_1$ is a $t$-greedy set for $x$ of cardinality $n_{k_0}$, we have
\begin{equation}
\left|\left|P_{A_1}(x)\right|\right|\le \C_{q,t}\|x\|.\label{tgreedysetA1}
\end{equation}
For every $2\le i\le j$, choose $ A_i\subset D_i$ such that $\left|D_i\right|=n_{k_0}$. Given that for every $2\le i\le j$, 
$$
\max_{m\in A_i}\left|\be_m^*(x)\right|\le \min_{m\in A_1}\left|\be_m^*(x)\right|,
$$
using convexity and the $\n$-UL and the $\n$-democracy properties we obtain
\begin{eqnarray}
\left|\left|P_{A_i}(x)\right|\right|&\le &\max_{m\in A_i}\left|\be_m^*(x)\right|\sup_{\bfe\in \Psi_{A_i}}\left|\left|\bff_{\bfe A_i}\right|\right|\le \min_{m\in A_1}\left|\be_m^*(x)\right|\sup_{\bfe\in \Psi_{D_i}}\left|\left|\bff_{\bfe D_i}\right|\right|\nonumber\\
&\le&  \min_{m\in A_1}\left|\be_m^*(x)\right|\C_2\left|\left|\bff_{ D_i}\right|\right|\le \C_2\Delta_d \min_{m\in A_1}\left|\be_m^*(x)\right|\left|\left|\bff_{A_1}\right|\right|\nonumber\\
&\le& \C_1\C_2 \Delta_d \left|\left|P_{A_1}(x)\right|\right|.\label{ndemb}
\end{eqnarray}
Combining this result with \eqref{tgreedysetA1} and using the triangle inequality, we get 
$$
\left|\left|P_{A}(x)\right|\right|\le \sum_{i=1}^{j}\left|\left|P_{A_i}(x)\right|\right|\le \C_{q,t}\left(1+\left(l-1\right)\C_1\C_2 \Delta_d\right)\|x\|. 
$$
This proves that $\B$ is $t$-quasi-greedy with constant as in the statement. To prove that it is symmetric for largest coefficients, we apply Proposition~\ref{propositionequivnslcnsdnqglc} and  Theorem~\ref{theoremnsymlc}, considering that $\B$ is $\C_{ql}$-$\n$-quasi-greedy for largest coefficients and $\Delta_s$-$\n$-superdemocratic, with $\C_{ql}\le \C_{q,t}$, and $\Delta_s\le \C_1\C_2\Delta_d$. \\
\ref{nsuperdemocratic} This is proven by essentially the same argument as the previous case. The only differences are that instead of \eqref{ndemb}, we obtain
 \begin{eqnarray*}
\left|\left|P_{A_i}(x)\right|\right|&=&\max_{m\in A_i}\left|\be_m^*(x)\right|\sup_{\bfe\in \Psi_{A_i}}\left|\left|\bff_{\bfe A_i}\right|\right|\le \min_{m\in A_1}\left|\be_m^*(x)\right|\sup_{\bfe\in \Psi_{D_i}}\left|\left|\bff_{\bfe D_i}\right|\right|\\
&\le&  \min_{m\in A_1}\left|\be_m^*(x)\right|\Delta_s\left|\left|\bff_{A_1}\right|\right|\\
&\le& \C_1 \Delta_s \left|\left|P_{A_1}(x)\right|\right|
\end{eqnarray*}
(and thus, we also get $\Delta_s$ instead of $\C_2\Delta_d$ in the upper bound for $\C$), and that we apply Proposition~\ref{propositionequivnslcnsdnqglc} using the hypothesis that $\B$ is $\Delta_s$-$\n$-superdemocratic. 
\end{proof}

\section{Examples}\label{sectionexamples}

In this section, we consider two families of examples that are used throughout the paper, and study the relevant properties of the bases. First, we construct a family of examples that proves that for sequences with arbitrarily large additive gaps, $\n$-unconditionality for constant coefficients, the $\n$-UL property,  $\n$-(super)democracy and $\n$-(super)-conservativeness are not equivalent to their standard counterparts.

\begin{example}\label{examplendemocracylargelineargaps} Given $\n$ with arbitrarily large additive gaps, choose recursively a subsequence $(n_{k_i})_{i\in \N_{0}}$ and $(m_i)_{i\in\N}$ a sequence of positive integers with $m_1>4$ so that for every $i\in \N$, 
\begin{equation}
m_i n_{{k_i}+1}^3<m_{i+1}, \qquad m_i^2< n_{k_{i}} \qquad n_{k_i}+2m_i<n_{k_{i}+1} \qquad\text{and}\qquad n_{{k_i}+1}^2<n_{k_{i+1}}, \label{recursive}
\end{equation}
and choose a sequence of sets of positive integers $(B_i)_{i\in \N}$ so that 
$$
m_i<B_{i}<B_{i+1}\qquad\text{ and}\qquad |B_i|=_{n_{k_i}+m_i}\qquad\forall i\in \N. 
$$
For each $i\in\N$, define
$$
\mathcal{F}_{i}:=\left\lbrace f=(f_j)_{j\in \N}\colon |\supp(f)|\le n_{{k_i}+1},\;\|f\|_{\infty}\le 1,\; \|f\|_1\le  m_i,\;\sum_{j\in B_i}f_j=0\right\rbrace
$$
and, for $(a_j)_{j\in\N}\in \mathtt{c}_{00}$, 
$$
\left\Vert (a_j)_{j\in \N}\right\Vert_{\diamond,i}:=\frac{1}{n_{{k_{i-1}+1}}^2}\sup_{f\in \mathcal{F}_i}\left\vert \sum_{j\in \N}f_ja_j\right\vert.
$$
Let $\X$ be the completion of $\mathtt{c}_{00}$ with the norm
$$
\|x\|=\max\left\lbrace\|x\|_{\infty},\|x\|_{\diamond}:=\sup_{i\in \N}\|x\|_{\diamond,i}\right\rbrace, 
$$
and let $\B$ be the canonical unit vector system of $\X$. Then, the following hold: 
\begin{enumerate}[a)]
\item \label{markushevich4}$\B$ is a normalized Markushevich basis for $\X$ with normalized biorthogonal funcitionals $\B^*$. 
\item \label{nsuperdemnew} $\B$ is $\C$-$\n$-superdemocratic, with $\C\le 2$. 
\item \label{nULnew} $\B$ has the $\n$-UL property, with $\max\{\C_1,\C_2\}\le 2$
\item \label{notucc} $\B$ is not $\n$-suppression unconditional for constant coefficients, and thus not unconditional for constant coefficients. 
\item \label{notconservative}$\B$ is not conservative. 
\end{enumerate}
\end{example}
\begin{proof}
\noindent\textbf{ Step \ref{markushevich4}:} It is clear that $\B$ and $\B^*$ are normalized. To see that $\B$ is a Markushevich basis, fix $x\in \X$ such that $\be_j^*(x)=0$ for all $j\in \N$, and choose a sequence $(x_l)_{l\in \N}$ with $x_l\in [\be_j: 1\le l\le s(l)]$ for some $s(l)\in \N$, and 
$$
x_l\xrightarrow[l\to \infty]{}x.
$$
Given $\nu>0$, choose $l_0\in \N$ so that 
$$
\|x_l-x\|\le \nu \qquad\forall l\ge l_0. 
$$
Now pick $i_0$ and $f\in \mathcal{F}_{i_0}$ so that 
$$
\left \Vert x_{l_0}\right\Vert_{\diamond} \le \nu+\|x_{l_0}\|_{\diamond,i_0}\le 2\nu + \sum_{j\in \N}f_j\be_j^*\left(x_{l_0}\right).
$$
Since $f$ has finite support and $\be_j^*(x)=0$ for all $j\in \N$, there is $l_1>l_0$ such that
\begin{equation}
\sum_{1\le j\le \max(\supp(f))+s\left(l_0\right)}\left \vert \be_j^*\left(x_{l_1}\right) \right\vert \le \nu. \label{lessthannu}
\end{equation}

Hence, 
$$
\sum_{j \in \N}f_j\be_j^*\left(x_{l_0}\right)\le \left\vert \sum_{j \in \N}f_j\be_j^*\left(x_{l_1}\right)\right\vert+  \left\vert \sum_{ j\in \N}f_j\be_j^*\left(x_{l_0}-x_{l_1}\right)\right\vert\le \nu +\left\Vert x_{l_0}-x_{l_1} \right\Vert\le 2\nu. 
$$
It follows that
$$
\left\Vert x_{l_0}\right\Vert _{\diamond}\le 4\nu. 
$$
Also by \eqref{lessthannu}, 
$$
\|x_{l_0}\|_{\infty}=\sup_{j\in \supp\left(x_{l_0}\right)}\left\vert \be_j^*\left(x_{l_0}\right) \right\vert\le \sup_{j\in \supp\left(x_{l_0}\right)}\left\vert \be_j^*\left(x_{l_1}\right) \right\vert +\sup_{j\in \supp\left(x_{l_0}\right)}\left\vert \be_j^*\left(x_{l_1}-x_{l_0}\right) \right\vert \le 3\nu. 
$$
We deduce that 
$$
\left\Vert x\right\Vert \le \nu +\left\Vert x_{l_0}\right\Vert \le 5\nu.
$$
Since $\nu$ is arbitrary, this entails that $x=0$ and completes the proof of \ref{markushevich4}. \\
To prove the rest of the statements, first we show the following: 

\begin{enumerate}[i.]

\item \label{ULpar1} For all $i\in \N$, all sets $A\subset \N$ with $1 \le |A|\le n_{{k_i}+1}$ and all scalars $(a_j)_{j\in A}$, 
\begin{equation}
\left\Vert \sum_{j\in A}a_j\be_j \right\Vert_{\diamond}\le \max_{j\in A}|a_j|\frac{m_i}{n_{{k_{i-1}}+1}^2}.\nonumber
\end{equation}

\item \label{ULpar1.5} For all $i\in \N$, all sets $A\subset \N$ with $m_i \le |A|\le n_{{k_i}}$ and all scalars $(a_j)_{j\in A}$, 
\begin{equation}
\left\Vert \sum_{j\in A}a_j\be_j \right\Vert_{\diamond}\ge \frac{1}{2} \min_{j\in A}|a_j|\frac{m_i}{n_{{k_{i-1}}+1}^2}.\nonumber
\end{equation}

\item \label{ULpar2} For all $i\in \N_{\ge 2}$, all sets $A\subset \N$ with $n_{{k_{i-1}}+1}\le |A|\le m_{i}$ and all scalars $(a_j)_{j\in A}$, 
\begin{equation*}
\min_{j\in A}|a_j| \max\left\lbrace \frac{|A|}{2 n_{{k_{i-1}}+1}^2},\frac{m_{i-1}}{n_{{k_{i-2}}+1}^2}  \right\rbrace\le \left\Vert \sum_{j\in A}a_j\be_j \right\Vert_{\diamond}\le \max_{j\in A}|a_j| \max\left\lbrace \frac{|A|}{ n_{{k_{i-1}}+1}^2},\frac{m_{i-1}}{n_{{k_{i-2}}+1}^2}  \right\rbrace.
\end{equation*}
\item \label{ULpar3} For all $A\subset \N$ with $1\le |A|\le m_1$ and all scalars $(a_j)_{j\in A}$,  
\begin{equation*}
\min_{j\in A}|a_j|\frac{|A|}{2 n_{{k_{0}}+1}^2}\le \left\Vert \sum_{j\in A}a_j\be_j \right\Vert_{\diamond}\le \max_{j\in A}|a_j| \frac{|A|}{ n_{{k_{0}}+1}^2}.
\end{equation*}

\end{enumerate}
To prove \ref{ULpar1}, suppose first that $l>i$ and $f\in \mathcal{F}_{l}$. By \eqref{recursive} we get
\begin{equation*}
\frac{1}{{n_{{k_{{l-1}}+1}}^2}}\left\vert \sum_{j\in A}f_j a_j\right\vert\le \max_{j\in A}|a_j| \frac{|A|}{{n_{{k_{{l-1}}+1}}^2}}\le  \max_{j\in A}|a_j| \frac{n_{{k_i}+1}}{{n_{{k_{{i}}+1}}^2}}\le \max_{j\in A}|a_j|\frac{m_i}{n_{{k_{i-1}+1}}^2}. 
\end{equation*}
Similarly, for each $l<i$ and each $f\in \mathcal{F}_{l}$ we have 
\begin{equation*}
\frac{1}{n_{{k_{l-1}+1}}^2}\left\vert \sum_{j\in A}f_j a_j\right\vert\le \max_{j\in A}|a_j| \frac{m_l}{n_{{k_{l-1}+1}}^2}\le \max_{j\in A}|a_j|\frac{m_i}{n_{{k_{i-1}+1}}^2}. 
\end{equation*}
Since, for $f\in \mathcal{F}_i$,
\begin{equation*}
\frac{1}{n_{{k_{i-1}+1}}^2}\left\vert \sum_{j\in A}f_j a_j\right\vert\le \max_{j\in A}|a_j|\frac{1}{n_{{k_{i-1}+1}}^2}\sum_{j\in \N}|f_j|\le \max_{j\in A}|a_j|\frac{m_i}{n_{{k_{i-1}}+1}^2},
\end{equation*}
taking supremum we complete the proof of~\ref{ULpar1}.\\
Next, we prove \ref{ULpar1.5}:  Assume $a_j\not=0$ for all $j\in A$, choose $A_1 \subset A$ and $A_2\subset B_i\setminus A$ with $|A_1|=|A_2|=m_i$, and let 
\begin{equation*}
f_j:=\begin{cases}
\frac{|a_j|}{2 a_j}& \text{if } j\in A_1;\\
-\frac{1}{m_i} \sum_{l\in A_1\cap B_i}f_l & \text{if } j\in A_2;\\
0& \text{in any other case.}
\end{cases}
\end{equation*}
Then $f=(f_j)_{j\in \N}\in \mathcal{F}_i$, and 
$$
\left\Vert \sum_{l\in A}a_l\be_l\right\Vert_{\diamond}\ge \frac{1}{n_{{k_{i-1}+1}}^2}\left\vert  \sum_{j\in \N}f_j\be_j^*\left(\sum_{l\in A}a_l\be_l\right)\right\vert =\frac{1}{2 n_{{k_{i-1}+1}}^2}\sum_{j\in A_1}|a_j|\ge \frac{m_i}{2n_{{k_{i-1}+1}}^2} \min_{j\in A}|a_j|.
$$

To prove \ref{ULpar2}, by a density argument we may assume $a_j\not=0$ for all $j\in A$. For every $l\ge i$ and every $f\in \mathcal{F}_l$,
\begin{equation*}
\frac{1}{n_{{k_{l-1}+1}}^2}\left\vert \sum_{j\in A}f_j a_j\right\vert\le \max_{j\in A}|a_j|\frac{|A|}{n_{{k_{i-1}+1}}^2}.
\end{equation*}
Hence, 
\begin{equation}
\sup_{l\ge i}\left\Vert \sum_{j\in A}a_j\be_j\right\Vert_{\diamond,l}\le \max_{j\in A}|a_j|\frac{|A|}{n_{{k_{i-1}+1}}^2}.\label{lgei}
\end{equation}
Now pick $B\subset B_i\setminus A$ with $|B|=|A|$, and define \begin{equation*}
f_j=\begin{cases}
\frac{|a_j|}{2 a_j}& \text{if } j\in A;\\
-\frac{1}{ |A|}\sum_{l\in A \cap B_i}f_l &\text{if } j\in B;\\
0 & \text{in any other case.} 
\end{cases}
\end{equation*}
Then $f=(f_j)_{j\in \N}\in \mathcal{F}_i$, and 
$$
\sum_{j\in\N}f_j\be_j^*\left(\sum_{l\in A}a_l\be_l\right)= \sum_{j\in A}f_j a_j=\frac{1}{2}\sum_{j\in A}|a_j|. 
$$
It follows from this and \eqref{lgei} that 
\begin{equation}
\frac{|A|}{2 n_{{k_{i-1}}+1}^2}\min_{j\in A}|a_j|\le \frac{1}{2 n_{{k_{i-1}}+1}^2}\sum_{j\in A}|a_j|\le \sup_{l\ge i}\left\Vert  \sum_{j\in A}f_j \be_j\right\Vert_{\diamond,i}\le  \max_{j\in A}|a_j|\frac{|A|}{n_{{k_{i-1}+1}}^2}. \label{lgeib}
\end{equation}

On the other hand, if $1\le l<i$, using \eqref{recursive} we obtain
\begin{eqnarray}
\frac{1}{n_{{k_{l-1}+1}}^2}\left\vert \sum_{j\in A}f_j a_j\right\vert&\le& \max_{j\in A}|a_j|\frac{m_l}{n_{{k_{l-1}+1}}^2}\le \max_{j\in A}|a_j|\frac{m_{i-1}}{n_{{k_{i-2}+1}}^2}.\label{llei}
\end{eqnarray}

Now pick $A_1\subset A\setminus B_{i-1} $ with $|A_1|=m_{i-1}$, and set
\begin{equation*}
f_j=\begin{cases}
\frac{|a_j|}{ a_j}& \text{if } j\in A_1;\\
0 & \text{otherwise.} 
\end{cases}
\end{equation*}
Then $f=(f_j)_{j\in \N}\in \mathcal{F}_{i-1}$, and 
$$
\sum_{j\in\N}f_j\be_j^*\left(\sum_{l\in A}a_l\be_l\right)= \sum_{j\in A}f_j a_j=\sum_{j\in A_1}|a_j|\ge m_{i-1}\min_{j\in A}|a_j|
$$
which, when combined with \eqref{llei} gives 
\begin{equation*}
\min_{j\in A}|a_j|\frac{m_{i-1}}{n_{{k_{i-2}+1}}^2}\le  \sup_{1\le l<i}\left\Vert  \sum_{j\in A}f_j \be_j\right\Vert_{\diamond,i}\le  \max_{j\in A}|a_j|\frac{m_{i-1}}{n_{{k_{i-2}+1}}^2}.
\end{equation*}
The proof of \ref{ULpar2} is now completed combining the above inequality with \eqref{lgeib}, whereas \ref{ULpar3} is proven by the same argument that gives \eqref{lgeib}. \\
\noindent\textbf{Step \ref{nsuperdemnew} $\n$-superdemocracy:}, fix $A,B\subset\N$ with $|A|=|B|=n\in \n$, and $\varepsilon\in A$, $\varepsilon'\in B$. Then $\|\bff_{\bfe A}\|\le 2\|\bff_{\bfe' B}\|$ is obtained as follows: 
\begin{itemize}
\item If there is $l\in \N$ such that $n_{k_{l}+1}\le n\le m_{l+1}$, apply \ref{ULpar2} with $i=l+1$. 
\item If there is $l\in \N$ such that $m_l\le n\le n_{k_l}$, combine \ref{ULpar1} and  \ref{ULpar1.5}.  
\item If $n\le m_1$, apply \ref{ULpar3}. 
\end{itemize}
\noindent\textbf{Step \ref{nULnew} $\n$-UL property:} This is proven in the same manner as Step \ref{nsuperdemnew}. \\

\noindent\textbf{Step \ref{notucc}  $\n$-suppression unconditionality for constant coefficients:} Fix $i>1$, and choose sets $ D_i\subset B_i$ with $|D_i|=n_{k_{i}}$. Then by \ref{ULpar1.5}, 
\begin{equation}
\|\bff_{D_i}\|\ge \frac{m_i}{2n_{{k_{i-1}+1}}^2}. \label{Bidifferentsigns}
\end{equation} 
Let us show that
\begin{equation}
\|\bff_{B_i}\|\le \frac{m_{i-1}}{n_{{k_{i-2}+1}}^2}.\label{Bismallnorm}
\end{equation}
For $1\le l< i$ and $f\in \mathcal{F}_l$, we have 
$$
\frac{1}{n_{{k_{l-1}+1}}^2}\left\vert \sum_{j\in \N}f_j\be_j^*\left(\bff_{B_i}\right)\right\vert\le\frac{m_l}{n_{{k_{l-1}+1}}^2}.
$$
Hence,
\begin{equation*}
\sup_{1\le l<i}\|\bff_{B_i}\|_{\diamond,i}\le \max_{1\le l\le i-1}\frac{m_l}{n_{{k_{l-1}+1}}^2}=\frac{m_{i-1}}{n_{{k_{i-2}+1}}^2} \qquad\text{ (by \eqref{recursive})}.
\end{equation*}
On the other hand, for $l>i$ and $f\in \mathcal{F}_l$, 
$$
\frac{1}{n_{{k_{l-1}+1}}^2}\left\vert \sum_{j\in \N}f_j\be_j^*\left(\bff_{B_i}\right)\right\vert\le \frac{|B_i|}{n_{{k_{l-1}+1}}^2}\le \frac{1}{n_{k_{l-1}+1}}\le 1. 
$$
Thus, 
\begin{equation*}
\sup_{l>i}\|\bff_{B_i}\|_{\diamond,i}\le 1. 
\end{equation*}
Given that by construction $\|\bff_{B_i}\|_{\diamond,i}=0$, \eqref{Bismallnorm} is proven, and it follows from that, \eqref{recursive} and \eqref{Bidifferentsigns} that 
$$
\frac{\|\bff_{D_i}\|}{\|\bff_{B_i}\|}\xrightarrow[i\to \infty]{}\infty, 
$$
so $\B$ is not $\n$-suppression unconditional for constant coefficients. \\
\textbf{Step \ref{notconservative} conservativeness: } For each $i\ge 2$, choose $E_i<B_i$ with $|E_i|=m_i$. By \ref{ULpar2}, 
$$
\|\bff_{E_i}\|\ge \frac{m_i}{2 n_{{k_{i-1}+1}}^2}.
$$
From this, \eqref{recursive} and \eqref{Bismallnorm} it follows that 
$$
\frac{\|\bff_{E_i}\|}{\|\bff_{B_i}\|}\xrightarrow[i\to \infty]{}\infty, 
$$
so $\B$ is not conservative.

\end{proof}

Next, we consider a family of examples from \cite[Proposition 3.1]{O2015}, with a slight modification for our purposes. 
\begin{example}\label{examplendemocratic}
Suppose $\n$ has arbitrarily large gaps, write $\n=(n_k)_{k=1}^\infty$ and find $k_1<k_2<...$ such that the sequence $(n_{k_i+1}/n_{k_i})_{i=1}^\infty$ increases without a bound and $n_{k_1}>4$. For $i\in\mathbb N$, write
	$$c_i=\left(\frac{n_{k_i+1}}{n_{k_i}}\right)^{1/4},\qquad m_i=\floor{\sqrt{n_{k_i+1}n_{k_i}}}.$$
	Let $\tilde{m}_i = \sum_{j<i}m_i$ (so that $\tilde{m}_1=0$ and $\tilde{m}_{i+1}=\tilde{m}_i+m_i$ for $i\geq 1$), $\beta_i:=\floor{\frac{m_i}{2}}$, and  let $\X$ be completion of $\mathtt{c}_{00}$ with the norm:  
		
	$$\left\Vert\sum_{j}a_j\be_j\right\Vert = \max\left\lbrace \Vert (a_j)_j\Vert_2, \sup_{i\in\mathbb N}\frac{c_i}{\sqrt{m_i}}\max_{i\leq l\leq m_i}\left\vert \sum_{j=\tilde{m}_i+1}^{\tilde{m}_i+l}(-1)^{\theta(j)}a_j\right\vert\right\rbrace,$$
where 
$$
\theta(j)=\begin{cases} 2j & \text{ if }\widetilde{m}_i+1\le j\le \widetilde{m}_i+\beta_i\\
j& \text{ if } \widetilde{m}_i+\beta_i+1\le j \le \widetilde{m}_i+m_i. 
\end{cases}
$$	
The unit vector basis $\mathcal B=(\be_i)_{i\in\mathbb N}$ is a monotone Schauder basis with the following properties. 
\begin{enumerate}[a)]
	\item \label{aex1}$\B$ is $\n$-$t$-quasi-greedy with $\C_{q,t}\le\frac{2}{t}$ for all $0<t\le 1$, and not quasi-greedy.
		\item  \label{bex1} $\B$ is $\Delta_s$-$\n$-superdemocratic with $\Delta_s \le \sqrt{2}$. 	
	\item  \label{dex1} $\B$ is $\Delta$-$\n$-symmetric for largest coefficients with $\Delta\le 3+2\sqrt{2}$.
		\item  \label{eex1} $\B$ has the $\n$-UL property with $\max\{\C_1,\C_2\}\le \sqrt{2}$. 
	
	\item  \label{fex1} $\B$ is not conservative.

\item \label{hex1} $\B$ is not unconditional for constant coefficients. Hence, it does not have the UL property. 
\end{enumerate}
\end{example}

\begin{proof}
It is clear from the definition that $\B$ is a monotone Schauder basis. \\
\noindent\textbf{Step \ref{aex1} $\n$-quasi-greediness:} This was proven in \cite[Proposition 3.1]{O2015}: The only modification introduced in our construction is that for some $j\in \N$, $\be_j$ is replaced with $-\be_j$, and it is clear that this change does not affect the $\n$-quasi-greedy or quasi-greedy properties.  \\
	\noindent\textbf{Step \ref{bex1} $\n$-superdemocracy:} Note that for every $m\in \N$, $2\floor*{\sqrt{m}}\ge \sqrt{m}$, so 
	\begin{equation}
	\frac{1}{\sqrt{\floor*{\sqrt{m}}}}=\frac{\sqrt{2}}{\sqrt[4]{m}}.\nonumber
	\end{equation}
	Now fix $B\subset \N$ with $|B|\in \mathbf{n}$, and  $\bfe \in \Psi_B$. For every $i\in \N$ with $|B|\le n_{k_i}$, we have
	\begin{eqnarray}
	\frac{c_i}{\sqrt{m_i}}\max_{1\le l\le m_i}\left|\sum\limits_{j=\widetilde{m}_i+1}^{\widetilde{m}_i+l }(-1)^{\theta(j)}\be_j^*(\mathbf{1}_{\bfe B})\right|&\le& \frac{c_i}{\sqrt{m_i}}|B|=\sqrt[4]{\frac{n_{k_i+1}}{n_{k_i}}}\frac{|B|}{\sqrt{\floor*{\sqrt{n_{k_i}n_{k_{i}+1}}}}}\nonumber\\
	&\le& \sqrt{2}\sqrt[4]{\frac{n_{k_i+1}}{n_{k_i}}}\frac{|B|}{\sqrt[4]{n_{k_i}n_{k_{i}+1}}}=\frac{\sqrt{2}|B|}{\sqrt{n_{k_i}}}\le \frac{\sqrt{2}|B|}{\sqrt{|B|}}\nonumber\\
	&=&\sqrt{2}\sqrt{|B|}.\label{biggeri}
	\end{eqnarray}
	On the other hand, if $|B|\ge n_{{k_i}+1}$, then 
	\begin{eqnarray}
	\frac{c_i}{\sqrt{m_i}}\max_{1\le k\le m_i}\left|\sum\limits_{j=\widetilde{m}_i+1}^{\widetilde{m}_i+l }(-1)^{\theta(j)}\be_j^*(\mathbf{1}_{\bfe B})\right|&\le& \frac{c_i}{\sqrt{m_i}}m_i= \sqrt[4]{\frac{n_{k_i+1}}{n_{k_i}}}\sqrt{\floor*{\sqrt{n_{k_i}n_{k_{i}+1}}}} \nonumber\\
	&\le& \sqrt[4]{\frac{n_{k_i+1}}{n_{k_i}}}\sqrt{\sqrt{n_{k_i}n_{k_{i}+1}}}=\sqrt{n_{k_i+1}}\le \sqrt{|B|}.\label{smalleri}
	\end{eqnarray}
	Taking supremum in \eqref{biggeri} and \eqref{smalleri} we get 
	\begin{equation}
	\|\mathbf{1}_{\bfe B}\|\le \sqrt{2}\sqrt{|B|}.\label{forbidem}
	\end{equation}
	As
	$$
	\|\mathbf{1}_{\bfe B}\|\ge \|\mathbf{1}_{\bfe B}\|_{2}=\sqrt{|B|},
	$$
	it follows that $\B$ is $\Delta_s$-$\n$-superdemocratic with $\Delta_s\le \sqrt{2}$.

	\noindent\textbf{Step \ref{dex1} $\n$-symmetry for largest coefficients:} It follows by \ref{aex1} that $\B$ is $\C_{ql}$-$\n$-quasi-greedy for largest coefficients with $\C_{ql}\le 2$. From that and \ref{bex1}, an application of Proposition~\ref{propositionequivnslcnsdnqglc} gives the desired result.

	\noindent\textbf{Step \ref{eex1} $\n$-UL property:} Fix $A\subset \N$ with $|A|\in \mathbf{n}$, and scalars $(a_i)_{i\in A}$. By convexity and using that the basis is $\sqrt{2}$-$\n$-superdemocratic, 
	\begin{equation}
	\|\sum\limits_{i\in A}a_i\be_i\|\le \max_{i\in A}|a_i|\max_{\bfe \in \Psi_A}\|\mathbf{1}_{\bfe A}\|\le \sqrt{2}\max_{i\in A}|a_i|\|\mathbf{1}_{A}\|.\label{rightside2}
	\end{equation}
	On the other hand, using \eqref{forbidem} we get
	\begin{eqnarray}
	\|\sum\limits_{i\in A}a_i\be_i\|&\ge& \|\sum\limits_{i\in A}a_i\be_i\|_2=\sqrt{\sum\limits_{i\in A}|a_i|^2}\ge \min_{j\in A}|a_j|\sqrt{|A|}\nonumber\\
	&\ge& \frac{1}{\sqrt{2}}\min_{i\in A}|a_i|\|\mathbf{1}_{A}\|. 
	\end{eqnarray}

	\noindent\textbf{Step \ref{fex1} conservativeness:} To see that $\B$ is not conservative, for each $i\in \N$ let 
	
		\begin{equation*}
	B_i:=\{\widetilde{m}_i+1,\dots, \widetilde{m}_i+\beta_i\}\qquad\text{and}\qquad D_i:=\{\widetilde{m}_i+\beta_i+1,\dots, \widetilde{m}_i+2\beta_i\}.
	\end{equation*}
	
	We have
	\begin{equation}
	\|\mathbf{1}_{B_i}\|\ge \left\vert \frac{c_i}{\sqrt{m_i}}\sum\limits_{j=\widetilde{m}_i+1}^{\widetilde{m}_i+\beta_i}(-1)^{\theta(j)}\be_j^*(\mathbf{1}_{B_i})\right\vert=\frac{c_i}{\sqrt{m_i}}\beta_i\ge \frac{c_i}{\sqrt{m_i}}\frac{m_i}{3}=\frac{c_i\sqrt{m_i}}{3}.\label{Bi}
	\end{equation}
On the other hand, for each $1\le l\le \beta_i$, 
$$
\sum\limits_{j=\widetilde{m}_i+1}^{\widetilde{m}_i+l}(-1)^{\theta(j)}\be_j^*(\mathbf{1}_{D_i})=0, 
$$
whereas for $\beta_{i}+1\le l\le m_i$, 
$$
\left|\sum\limits_{j=\widetilde{m}_i+1}^{\widetilde{m}_i+l}(-1)^{\theta(j)}\be_j^*(\mathbf{1}_{D_i})\right|= \left|\sum\limits_{j=\widetilde{m}_i+\beta_i+1}^{\widetilde{m}_i+l}(-1)^{j}\be_j^*(\mathbf{1}_{D_i})\right|\le 1. 
$$
Since 
$$
\sum\limits_{j=\widetilde{m}_{i'}+1}^{\widetilde{m}_{i'}+l}(-1)^{\theta(j)}\be_j^*(\mathbf{1}_{D_i})=0\qquad\forall i'\not=i \forall 1\le l\le m_{i'},
$$
we deduce that
\begin{equation}
\|1_{D_i}\|= \|1_{D_i}\|_2\le  \sqrt{\beta_i+1}\le \sqrt{m_i}.\label{notconservative4}
\end{equation}
Given that $(c_i)_i$ is unbounded, $B_i<D_i$ and $|B_i|\le |D_i|$ for all $i$, it follows from \eqref{Bi} and \eqref{notconservative4} that $\B$ is not conservative. \\

\noindent\textbf{Step \ref{hex1}  Unconditionality for constant coefficients:} This can be proven using the argument given in  \cite[Proposition 3.2]{O2015} to prove that the basis is not quasi-greedy. We give a proof for the sake of completion: Fix $i\in \N$, and consider again the set $B_i$. By \eqref{Bi}, we have
$$
\|\bff_{ B_i}\|\ge \frac{c_i\sqrt{m_i}}{3}.
$$
Now let $\varepsilon\in \Psi_{B_i}$ be a sequence of alternating signs. Then for all $1\le l\le m_i$ we have 
$$
\frac{c_i}{\sqrt{m_i}}\left|\sum\limits_{j=\widetilde{m}_i+1}^{\widetilde{m}_i+l}(-1)^{\theta(j)}\be_j^*(\mathbf{1}_{\bfe B_i})\right|=\frac{c_i}{\sqrt{m_i}}\left|\sum\limits_{j=\widetilde{m}_i+1}^{\widetilde{m}_i+\max\{l,\beta_i\}}\varepsilon_j\right|\le \frac{c_i}{\sqrt{m_i}}\le 2.
$$
As 
$$
\frac{c_{i'}}{\sqrt{m_{i'}}}\left|\sum\limits_{j=\widetilde{m}_l+1}^{\widetilde{m}_l+l}(-1)^{\theta(j)}\be_j^*(\mathbf{1}_{\bfe B_i})\right|=0 \qquad\forall i'\not=i\forall 1\le l\le m_{i'}, 
$$
it follows that 
$$
\|\bff_{\bfe B_i}\|=\|\bff_{\bfe B_i}\|_2=\sqrt{\beta_i}\le \sqrt{m_i}.
$$
As before, using the fact that $(c_i)_i$ is unbounded we conclude that $\B$ is not unconditional for constant coefficients. 
\end{proof}

\begin{remark}\label{remarkunconditional}\rm A slight modification of Example~\ref{examplendemocratic} shows that even for unconditional Schauder bases, $\n$-superdemocracy does not entail democracy, or even conservativeness. Indeed, if we replace the norm in Example~\ref{examplendemocratic} by the norm 
$$\left\Vert\sum_{j}a_j\be_j\right\Vert_{\diamond} = \max\left\lbrace \Vert (a_j)_j\Vert_2, \sup_{i\in\mathbb N}\frac{c_i}{\sqrt{m_i}} \sum_{j=\tilde{m}_i+1}^{\tilde{m}_i+m_i}\left|a_j\right|\right\rbrace,$$
the resulting basis is $1$-unconditional, and the proof of $\n$-superdemocracy  holds with only minor, strightforward modifications. Since  $(c_i)_i$ is unbounded, 
$$
\|\bff_{B_i}\|_{\diamond}\ge \|\bff_{B_i}\|\ge \frac{c_i\sqrt{m_i}}{3}\ge \frac{c_i}{3}\sqrt{|B_i|},
$$
and the subsequence $(\be_{\widetilde{m}_i+1})_{i\in\N}$ is clearly equivalent to the unit vector basis of $\ell_2$, $\B$ is not conservative. 
\end{remark}

\color{black}

\section*{Annex: Summary of some important constants}
\begin{table}[ht]
	\begin{center}
	\begin{tabular}{c c c}\hline
		
			& & \\
			{\bf Symbol}& {\bf Name of constant} &  {\bf Ref. equation}\\ \hline
			& & 
			\\
					$\C_{q}$ & Quasi-greedy constant & \eqref{q} \\ 
						& & \\
					$\mathbf{K}_u $ & Unconditionality for constant coeff. constant & \eqref{ucc} \\ 
			& & \\
			$\Delta_d$         & Democracy constant & \eqref{demo} \\
			& & \\
			$\Delta_s$ & Superdemocracy constant & \eqref{demo} \\ 
			& & \\
			$\Delta$ & Symmetry for largest coeff. constant &  \eqref{sy} \\ 
				\end{tabular}
	\end{center}

\end{table}

\end{document}